\documentclass[aap]{imsart}

\RequirePackage{amsmath,amsthm,amsfonts,amssymb}
\RequirePackage[numbers]{natbib}
\RequirePackage[colorlinks,citecolor=blue,urlcolor=blue]{hyperref}
\RequirePackage{graphicx}

\usepackage[utf8]{inputenc}
\usepackage{comment}
\usepackage{algorithm}
\usepackage[nocomma]{optidef}
\usepackage{enumitem}
\usepackage{booktabs}

\startlocaldefs
\theoremstyle{plain}
\newtheorem{theorem}{Theorem}[section]
\newtheorem{proposition}[theorem]{Proposition}
\newtheorem{lemma}[theorem]{Lemma}
\newtheorem{corollary}[theorem]{Corollary}

\theoremstyle{remark}
\newtheorem{definition}[theorem]{Definition}
\newtheorem{remark}[theorem]{Remark}

\newcommand{\bigCI}{\mathrel{\text{\scalebox{1.07}{$\perp\mkern-10mu\perp$}}}}

\DeclareMathOperator{\supp}{supp}

\DeclareMathOperator{\diam}{diam}

\DeclareMathOperator{\id}{id}

\DeclareMathOperator{\Var}{Var}

\DeclareMathOperator*{\argmin}{arg\, min}
\DeclareMathOperator*{\argmax}{arg\, max}

\newcommand{\1}{\mathbf{1}}

\newcommand{\mykill}[1]{}

\newcommand{\eins}{\1}

\newcommand{\X}{\mathcal{X}}
\newcommand{\Y}{\mathcal{Y}}
\newcommand{\Z}{\mathcal{Z}}
\newcommand{\indep}{\,\rotatebox[origin=c]{90}{$\models$}}

\endlocaldefs

\begin{document}

\begin{frontmatter}
\title{Computational methods for adapted optimal transport}
\runtitle{Computational methods for adapted optimal transport}

\begin{aug}
\author[A]{\fnms{Stephan} \snm{Eckstein}\ead[label=e1,mark]{stephan.eckstein@math.ethz.ch}}
\and
\author[A]{\fnms{Gudmund} \snm{Pammer}\ead[label=e2,mark]{gudmund.pammer@math.ethz.ch}}

\address[A]{Department of Mathematics, ETH Zürich,
	\printead{e1,e2}}

\end{aug}

\begin{abstract}
Adapted optimal transport (AOT) problems are optimal transport problems for distributions of a time series where 
couplings are constrained to have a temporal causal structure. In this paper, we develop computational tools for solving AOT problems numerically. 
First, we show that AOT problems are stable with respect to perturbations in 
the marginals and thus arbitrary AOT problems can be approximated by sequences of linear programs. We further study entropic 
methods to solve AOT problems. We show that any entropically regularized AOT problem 
converges to the corresponding unregularized problem if the regularization 
parameter goes to zero. The proof is based on a novel method---even in the 
non-adapted case---to easily obtain smooth approximations of a given coupling 
with fixed marginals. Finally, we show tractability of the adapted version of Sinkhorn's algorithm. We give explicit solutions for the 
occurring projections and prove that the procedure converges to the optimizer 
of the entropic AOT problem.
\end{abstract}



\end{frontmatter}

\section{Introduction}
For distributions of a time series, the usual Wasserstein distance does not account for the underlying information structure dictated by the flow of time. This can be inadequate in certain situations:  indeed, problems like optimal stopping, stochastic control 
or
pricing and hedging in mathematical finance, 
as well as
operations such as the Doob decomposition or the Snell envelop, are discontinuous with respect to the Wasserstein topology, and Wasserstein barycenters of martingale measures do not retain the martingale property, 
see \cite{backhoff2020adapted,bartl2021wasserstein}.
The adapted Wasserstein distance remedies these problems \cite{backhoff2020all,bartl2021wasserstein}. More generally, adapted optimal transport (AOT) problems are becoming increasingly popular as an alternative to classical optimal transport (OT) whenever a temporal component is included in the setting.
These problems range from the study of \textsl{dynamic} Cournot--Nash equilibria in games of mean field kind (see, e.g., \cite{acciaio2021cournot, backhoff2022dynamic}), generative models where temporal data is used (see, e.g., \cite{xu2021quantized, xu2020cot}), quantifying approximation quality in multistage stochastic programs (see, e.g., \cite{pflug2010version}), derivation of concentration inequalities on the Wiener space (see, e.g., \cite{lassalle2018causal}), robust pricing and hedging in mathematical finance (see, e.g., \cite{backhoff2020adapted}), and sensitivity analysis in multiperiod optimization problems (see, e.g., \cite{bartl2022sensitivity}).


\subsection{Adapted optimal transport problems}
For two marginal distributions $\mu$ and $\nu$, the classical optimal transport problem is of the form
\begin{align}
\tag{OT}
\inf_{\pi \in \Pi(\mu, \nu)} \int c(x, y) \,\pi(dx, dy),
\end{align}
where $c$ is a cost function and $\Pi(\mu, \nu)$ is the set of couplings between $\mu$ and $\nu$. We are interested in settings where $\mu$ and $\nu$ are distributions of a time series, say on $\mathbb{R}^N$,
\[
\mu \sim (X_1, X_2, \dots, X_N) ~ \text{ and } ~ \nu \sim (Y_1, Y_2, \dots, Y_N).
\]
The basic concept of AOT is most natural in its Monge formulation. Recall that the Monge formulation for classical OT is given by
\begin{align}
\label{eq:mongeot}
\tag{Monge-OT}
\inf_{T :~ T_* \mu = \nu} \int c(x, T(x)) \,\mu(dx),
\end{align}
where the infimum is taken over all measurable maps $T : \mathbb{R}^N \rightarrow \mathbb{R}^N$ satisfying the pushforward condition that $T_* \mu$ equals $\nu$. 
The basic premise of AOT is that these transport maps $T$ should be \emph{adapted}. While general maps can be written as $T(x) = (T_1(x), \dots, T_N(x))$, adapted maps are of the form
\begin{align}
\tag{adapted Monge maps}
T(x) = (T_1(x_1), T_2(x_1, x_2), \dots, T_N(x)),
\end{align}
i.e., $T_t$ only depends on $x_1, \dots, x_t$. Adapted transport maps encode the intuition that the decision of where to transport mass at time $t$ can only depend on the realizations of the values of $x$ that are known up to time $t$. Using adapted transport maps, we can define the adapted analogue to \eqref{eq:mongeot} as 
\begin{align}
\tag{Monge-AOT}
\label{eq:causal-monge}
\inf_{\substack{T :~ T_* \mu = \nu,\\ T \text{ adapted}}} \int c(x, T(x)) \,\mu(dx).
\end{align}
Going towards a Kantorovich formulation over couplings, consider the coupling $\pi \in \Pi(\mu, \nu)$ which is induced by $T$, i.e., $\pi \sim (X, T(X))$. Adaptedness of $T$ means that under $\pi$, the first $t$ components of $Y$ are independent of $(X_{t+1}, \dots, X_N)$ given $(X_1, \dots, X_t)$. This is a condition which can be transferred to arbitrary couplings, and we accordingly define
\begin{align*}
\Pi_c(\mu, \nu) = \{ \pi \in \Pi(\mu, \nu): \text{if } (X, Y) \sim \pi, \text{ then } (Y_1, \dots, Y_t) \indep_{(X_1, \dots, X_t)} (X_{t+1}, \dots, X_N)\}
\end{align*}
which we call the set of causal couplings between $\mu$ and $\nu$ (cf.~Definition \ref{def:causality} and Remark \ref{rem:equivdefcausal}). This leads to the first AOT problem, called causal OT, which is studied in this paper,
\begin{align}
\label{eq:causal-ot}
\tag{Causal-OT}
\inf_{\pi \in \Pi_c(\mu, \nu)} \int c(x, y) \,\pi(dx, dy).
\end{align}
Notably, \eqref{eq:causal-ot} is the \emph{correct} relaxation of \eqref{eq:causal-monge} in the sense that adapted Monge maps are dense in $\Pi_c(\mu,\nu)$ for suitable non-atomic measures $\mu$, see \cite{beiglbock2018denseness}.
While the problems \eqref{eq:causal-monge} and \eqref{eq:causal-ot} respect the temporal component, both problems lack symmetry.
In many situations, and particularly when going towards a concept of adapted Wasserstein distance, symmetry is desirable or even necessary.  
To alleviate asymmetry, one may symmetrize the causality constraint and define the set of bicausal couplings
\begin{align*}
\Pi_{bc}(\mu, \nu) = \{ \pi \in \Pi(\mu, \nu): \text{if } (X, Y) \sim \pi, \text{ then } (Y_1, \dots, Y_t) &\indep_{(X_1, \dots, X_t)} (X_{t+1}, \dots, X_N), \\
(X_1, \dots, X_t) &\indep_{(Y_1, \dots, Y_t)} (Y_{t+1}, \dots, Y_N) \}
\end{align*}
and the corresponding  problem
\begin{align}
\label{eq:bicausal-ot}
\tag{Bicausal-OT}
\inf_{\pi \in \Pi_{bc}(\mu, \nu)} \int c(x, y) \,\pi(dx, dy).
\end{align}
In the same spirit as \eqref{eq:causal-monge} and \eqref{eq:causal-ot}, one can also consider a Monge version of \eqref{eq:bicausal-ot}, c.f.\ \cite{beiglbock2022denseness}.
We call both \eqref{eq:causal-ot} and \eqref{eq:bicausal-ot} AOT problems and, if we can use both interchangeably, we simply write
\begin{align}
\tag{AOT}
\label{eq:aot}
\inf_{\pi \in \Pi_{\bullet}(\mu, \nu)} \int c(x, y) \,\pi(dx, dy),
\end{align}
where $\bullet$ is a placeholder for either symbol, $c$ or $bc$.


\subsection{Discretization and stability} 
With increasing theoretical interest in AOT problems, building a solid computational foundation is important to enable the utilization of AOT in applications. As a starting point, like classical OT, AOT reduces to a linear program whenever the marginal distributions are finitely supported, see Lemma \ref{lem:discretization}. Discrete bicausal problems even allow for a backward induction, reducing the optimization to iteratively solving families of classical OT problems, see \cite{backhoff2017causal}.

To solve general AOT problems with marginals $\mu$ and $\nu$, the baseline strategy is to approximate the marginals by a sequence of discrete (i.e.~finitely supported) marginals $\mu^n$ and $\nu^n$. Controlling the resulting approximation error requires a suitable stability of AOT problems with respect to perturbations in the marginals, which means the mapping
\begin{align*}
(\mu, \nu) &\mapsto \inf_{\pi \in \Pi_\bullet(\mu, \nu)} \int c(x, y) \,\pi(dx, dy)
\end{align*}
has to be continuous in a suitable sense. We provide both qualitative and quantitative continuity results for this mapping in Theorems \ref{thm:qual_stab} and \ref{thm:quantitative_stab}. The results extend to entropically regularized AOT problems (cf.~Section \ref{subsec:introentropic} below), and in this case we even obtain stability results for the respective optimizing couplings.
Our main technique is based on an adapted version of the concept of the \emph{shadow} -- a coupling which shadows the dependence structure of another coupling between differing marginals (see Definition \ref{def:shadow}) -- which builds on the recent work in \cite{eckstein2021quantitative}. A necessary component of our stability results is that marginals are compared with respect to an adapted Wasserstein distance, and not just a Wasserstein distance. This entails an important practical implication: when approximating an AOT problem with marginals $\mu$ and $\nu$ by a sequence of discrete problems, the marginals have to be approximated with respect to an adapted distance, and not just the Wasserstein distance. Such approximations are studied for instance in \cite{backhoff2020estimating,bartl2021wasserstein,beiglbock2018denseness}.

Our results show that AOT is surprisingly well behaved for an OT problem with additional constraints. This is not always the case, for instance, the related martingale OT problem is not stable in dimension $d\geq 2$ (at least for the weak topology on the marginals) as shown in \cite{bruckerhoff2021instability}, while the corresponding question in dimension $d=1$ received considerable attention \cite{backhoff2019stability, guo2019computational,wiesel2019continuity}. We further mention that computational methods for AOT problems have so far mostly been focused on backward induction, see \cite{pflug2010version,pflug2012distance,pichler2021nested}. A notable exception is \cite{xu2020cot}, where the authors utilize an approximate numerical scheme for causal OT based on duality.

\subsection{Entropic regularization}\label{subsec:introentropic} Many of the recent applications of OT build on fast approximate solutions which are based on regularization (e.g., \cite{cuturi2014fast,gulrajani2017improved,peyre2019computational}). Particularly entropic regularization is widely used, as it allows for a computationally efficient solution of the regularized problem using Sinkhorn's algorithm \cite{cuturi2013sinkhorn}. 
We study entropic regularization for general AOT problems in this paper. Entropic AOT problems are of the form
\begin{align}
\tag{Entropic-AOT}
\label{eq:entropic-aot}
\inf_{\pi \in \Pi_\bullet(\mu, \nu)} \int c \,d\pi + \varepsilon D_{\rm KL}(\pi, P),
\end{align}
where $\varepsilon > 0$, $D_{\rm KL}$ denotes the Kullback-Leibler divergence (or relative entropy) and $P= \mu \otimes \nu$ is the product measure between $\mu$ and $\nu$.
Utilization of entropic regularization for discrete bicausal problems was recently initiated by \cite{pichler2021nested}, where the authors used Sinkhorn's algorithm to speed up the computation of the backward induction, see also Remark \ref{rem:sink_back} for a short discussion of the approach in \cite{pichler2021nested} compared with the adapted version of Sinkhorn's algorithm (cf.~Section \ref{subsec:introsink}) studied in this paper. 
We showcase in Theorem \ref{thm:reg_conv} that for general AOT problems, similar to classical OT, the approximation error introduced by entropic regularization vanishes for a small enough regularization parameter, i.e.,
\[
\inf_{\pi \in \Pi_\bullet(\mu, \nu)} \int c \,d\pi + \varepsilon D_{\rm KL}(\pi, P) ~~\stackrel{\varepsilon \rightarrow 0}{\longrightarrow} ~~ \inf_{\pi \in \Pi_\bullet(\mu, \nu)} \int c \,d\pi.
\]
The proof is again based on the adapted version of the shadow, which is used in \mbox{Lemma \ref{lemma:shadowshadow}} to obtain smoothed versions of a given coupling. Even for classical OT, this approach may be a fruitful alternative to other methods studied for instance in \cite{bindini2020smoothing,carlier2017convergence,eckstein2020robust,leonard2012schrodinger}. For our purposes, this method has the benefit that the smoothed coupling inherits the adaptedness properties of the initial coupling. 

\subsection{Sinkhorn's algorithm}\label{subsec:introsink} We introduce and study an adapted version of Sinkhorn's algorithm to solve \eqref{eq:entropic-aot}. In this algorithm, one defines the initial measure $\pi^{(0)}$ via $\frac{d\pi^{(0)}}{dP} \propto \exp(-\frac{1}{\varepsilon} c)$, and iteratively 
\[
\pi^{(2k)} := \argmin_{\pi \in \mathcal{Q}_1} D_{\rm KL}(\pi, \pi^{(2k-1)}) ~ \text{ and } \pi^{(2k-1)} := \argmin_{\pi \in \mathcal{Q}_2} D_{\rm KL}(\pi, \pi^{(2k-2)})
\]
for $k \geq 1$.
Hereby, the sets $\mathcal{Q}_1$ and $\mathcal{Q}_2$ are chosen such that $\Pi_\bullet(\mu, \nu) = \mathcal{Q}_1 \cap \mathcal{Q}_2$ and the respective projection steps should be easy to compute. 
For instance for the causal problem \eqref{eq:causal-ot}, we choose $\mathcal{Q}_1 = \cup_{\tilde\nu} \Pi_c(\mu, \tilde\nu)$, the set of couplings with fixed first and arbitrary second marginal, and $\mathcal{Q}_2 = \cup_{\tilde\mu} \Pi(\tilde\mu, \nu)$, the set of couplings with arbitrary first and fixed second marginal.
We show computational tractability of the resulting projection steps, both in the primal and dual form, see Lemmas \ref{lem:proj_primal} and \ref{lem:proj_dual}, respectively. We further show that the values corresponding to $\pi^{(k)}$ converge, even linearly, to the optimal value of \eqref{eq:entropic-aot},
see Theorem \ref{thm:sink_lin_conv}.
The method of proof is based on the recent work \cite{carlier2021linear} proving convergence of Sinkhorn's algorithm for multimarginal OT. We give another short proof of convergence of Sinkhorn's algorithm in Proposition \ref{prop:conv_causal_sink} based on stability, which is however only applicable to the causal problem. In contrast to Theorem \ref{thm:sink_lin_conv}, Proposition \ref{prop:conv_causal_sink} does not give a rate of convergence, but it applies even to unbounded cost functions under suitable moment assumptions on the marginals.

\subsection{Duality}
Duality is a key concept to fully understand and work with optimal transport problems, which has been studied in the context of AOT by, e.g., \cite{acciaio2021cournot,backhoff2017causal}. However, particularly the dual formulations for entropic AOT, which are necessary tools to study the adapted version of Sinkhorn's algorithm, are not included in the current literature. In Proposition \ref{prop:duality}, we show that for bounded and continuous cost functions and suitably defined sets of functions $\mathcal{H}_\bullet \subset L^\infty$, we have the dual formulations:
\begin{align*}
\inf_{\pi \in \Pi_{\bullet}(\mu, \nu)} \int c \,d\pi&= \sup_{h \in \mathcal{H}_\bullet:\,h\leq c} \int h \,dP, \\
\inf_{\pi \in \Pi_{\bullet}(\mu, \nu)} \int c \,d\pi + \varepsilon D_{\rm KL}(\pi, P)&= \sup_{h \in \mathcal{H}_\bullet} \int h - \varepsilon \exp\Big(\frac{s-c}{\varepsilon}\Big) + \varepsilon \,dP.
\end{align*}
The proof builds on known duality results and Sion's minimax theorem.
For entropic AOT, we show in Lemma \ref{lem:schrodinger} that the dual supremum is attained and that optimizers can be characterized by adapted versions of the Schrödinger equations.  This result is a key component in the proof of Theorem \ref{thm:sink_lin_conv} showing the linear convergence of Sinkhorn's algorithm.

\subsection{Structure of the paper}
The remainder of this paper is structured as follows: in Section \ref{sec:notation} we introduce necessary notation and define the AOT problems. In Section \ref{sec:stab}, results related to the shadow and stability of AOT problems are derived. Section \ref{sec:conv_reg} is concerned with the convergence of the entropically regularized AOT problem for vanishing regularization parameter. Section \ref{subsec:duality} states and proves the dual formulations for AOT.  And finally, Section \ref{sec:sinkhorn} gives the results related to the adapted version of Sinkhorn's algorithm.

%
%
%
%
%

\section{Setting, notation and preliminary results}
\label{sec:notation}
\subsection{Notational conventions}
Let $N \in \mathbb{N}$ denote the number of time steps and write $\mathcal{X} = \mathcal{X}_1 \times \dots \times \mathcal{X}_N, \mathcal{Y} = \mathcal{Y}_1 \times \dots \times \mathcal{Y}_N$, $\mathcal Z = \mathcal Z_1 \times \dots \times \mathcal Z_N$ for Polish path spaces with metrics $d_{\X}, d_\Y, d_\Z$, respectively.
We denote by $C(\mathcal{X})$ (resp.~$C_b(\mathcal{X})$) the continuous (resp.~and bounded) functions mapping $\mathcal{X}$ to $\mathbb{R}$. Further, we say that a function $c : \X \times \Y \rightarrow \mathbb{R}$ has growth of order at most $p$ whenever $\frac{c}{1 + d_\X(\hat{x}, \cdot)^p + d_\Y(\hat{y}, \cdot)^p}$ is bounded for some $(\hat{x}, \hat{y}) \in \X \times \Y$. 
Moreover, let $L^\infty(\X)$ be the set of bounded and measurable functions mapping $\X$ to $\mathbb{R}$, and let $\mathcal{P}(\X)$ be the set of Borel probability measures on $\X$, and $\mathcal{P}_p(\X) \subset \mathcal{P}(\X)$ the measures with finite $p$-th moment (integrating $d_\X(\hat{x}, \cdot)^p$).
For all definitions stated above, one can replace $\mathcal X$ with any other Polish space.
We denote by $\Pi(\mu, \nu)$ the set of couplings between $\mu \in \mathcal{P}(\X)$ and $\nu \in \mathcal{P}(\Y)$, that is the set of probability measures on $\X \times \Y$ with first marginal $\mu$ and second marginal $\nu$, and we will always use the notation $P = \mu \otimes \nu$.

Since adapted transport problems involve disintegrations and taking marginals, we will take some time to introduce a proper notation:

First, for measures on $\X$ or $\Y$ (we give the notation for $\X$, while the corresponding versions for $\Y$ are analogous): Let $A \subset \{1, \dots, N\}$ and $\mu \in \mathcal{P}(\X)$. 
Define $\X_A := \otimes_{t\in A} \X_t$ and $X_A$ as the projection of $\X$ onto $\X_A$. 
Define $\mu_A := (X_A)_* \mu$, where $_*$ denotes the pushforward operator. 
That means, $\mu_A$ is the marginal of $\mu$ on the coordinates included in $A$. 
If $A = \{t\}$ we will write $\mu_A = \mu_t$ and if $A = \{1, \dots, t\}$ we will write $\mu_A = \mu_{1:t}$, and similarly with $X_t$, $X_{1:t}$, etc.
For the disintegration, we use the notation $\mu = \mu_A \otimes \mu^{X_A}$, where $\mu^{X_A} : \X_A \rightarrow \mathcal{P}(X_{\{1, \dots, N\}\backslash A})$ is a stochastic kernel. For $x_A \in \X_A$ we will abuse notation slightly and write $\mu^{x_A} := \mu^{X_A}(x_A)$. Finally for $B \subset \{1, \dots, N\}\backslash A$, we define $\mu^{X_A}_B := (X_B)_* (\mu^{X_A})$ (which is a stochastic kernel mapping $\X_A \rightarrow \mathcal{P}(\X_B)$), where the same notation $\mu^{x_A}_B = \mu^{X_A}_B(x_A)$ applies. The standard disintegration thus reads
\begin{align}
\label{eq:mu_disint}
\mu = \mu_1 \otimes \mu_2^{X_1} \otimes \mu_3^{X_{1:2}} \otimes \dots \otimes \mu_N^{X_{1:N-1}}.
\end{align}
Now, for joint distributions. Let $A, B \subset \{1, \dots, N\}$ and $\pi \in \mathcal{P}(\X \times \Y)$.
With slight abuse of notation, we regard $X_A$ and $Y_B$ also as projections mapping from $\X \times \Y$ onto the respective coordinates in $\X$ and $\Y$. 
Then, we define $\pi_{A, B} := (X_A, Y_B)_* \pi$ and as previously $\pi_{s, t}$ for $A = \{s\}, B = \{t\}$ and $\pi_{1:s, 1:t}$ for $A = \{1, \dots, s\}, B = \{1, \dots, t\}$.
For the disintegration, we write $\pi = \pi_{A, B} \otimes \pi^{X_{A}, Y_{B}}$ and for $A_2 \subset \{1, \dots, N\}\backslash A$ and $B_2 \subset \{1, \dots, N\}\backslash B$ we define $\pi^{X_A, Y_B}_{A_2, B_2} := (X_{A_2}, Y_{B_2})_* (\pi^{X_A, Y_B})$ (which is a stochastic kernel mapping $\X_A \times \Y_B$ onto $\mathcal{P}(\X_{A_2} \times \Y_{B_2})$. 
The same convention as previously applies when plugging values $x_A, y_B$ into the kernels. 
The special case $A \neq \emptyset = B$ (and similarly when the roles of $A$ and $B$ are reversed) is treated as follows: instead of writing the empty set we will use $0$ and set $(X_A,Y_B) := (X_A, Y_0) := X_A$.
For example, then $\pi_{A,0}$ denotes the law of $X_A$ and $\pi_{1:N,0}$ denotes the $X$-marginal of $\pi$.
Similarly, we write $\pi^{X_A,Y_B} = \pi^{X_A,Y_0}$ for the disintegration of $\pi$ w.r.t.\ $X_A$.
With these conventions in mind, the standard disintegration reads
\begin{align}
\pi &= \pi_{1, 1} \otimes \pi_{2, 2}^{X_1, Y_1} \otimes \pi_{3, 3}^{X_{1:2}, Y_{1:2}} \otimes \dots \otimes \pi_{N, N}^{X_{1:N-1}, Y_{1:N-1}} \nonumber\\
&= (\pi_{1, 0} \otimes \pi_{0, 1}^{X_1, Y_{0}}) \otimes (\pi_{2, 0}^{X_1, Y_1} \otimes \pi_{0, 2}^{X_{1:2}, Y_1}) \otimes \dots \otimes (\pi_{N, 0}^{X_{1:N-1}, Y_{1:N-1}} \otimes \pi_{0, N}^{X_{1:N}, Y_{1:N-1}}).
\label{eq:disint_temp}
\end{align}
Another disintegration which is perhaps more intuitive when considering $\pi$ as a randomized transport from $\X$ to $\Y$ would be
\begin{align}
\pi &= \pi_{1:N, 0} \otimes \pi_{0, 1:N}^{X_{1:N}, Y_0} \nonumber\\
&= (\pi_{1, 0} \otimes \pi_{2, 0}^{X_{1}, Y_0} \otimes \dots \otimes \pi_{N, 0}^{X_{1:N-1}, Y_0}) \otimes (\pi_{0, 1}^{X_{1:N}, Y_0} \otimes \pi_{0, 2}^{X_{1:N}, Y_1} \otimes \dots \otimes \pi_{0, N}^{X_{1:N}, Y_{1:N-1}}).
\label{eq:disint_transport}
\end{align}

Let $\mu \in \mathcal P(\X)$ and let $K \colon \X \to \mathcal P(\Y)$ be a measurable kernel.
Then we write $\mu K \in \mathcal P(\Y)$ for the measures that satisfies, for all $f \in C_b(\Y)$,
\begin{equation} \label{eq:def.muK} \int f(y) \, \mu K(dy) = \int f(y) \, \mu \otimes K(x)(dy).  \end{equation}
In order to reduce the number of brackets we will frequently write $K(x,dy)$ instead of $K(x)(dy)$. 

\subsection{Adapted optimal transport and its entropic versions}

\begin{definition}
	\label{def:causality}
	A measure $\pi \in \mathcal P(\X \times \Y)$ is called causal if, for all $t \in \{1,\ldots,N-1\}$, we have $\pi$-almost surely
	\begin{equation}
	\label{def:causal_property}
	\pi_{t+1,0}^{X_{1:t},Y_{1:t}} = \pi_{t+1,0}^{X_{1:t},Y_{0}},
	\end{equation}
	or equivalently, we have under $\pi$ the following conditional independence
	\begin{equation}
	\label{def:causal_property_random_variables}
	X_{t + 1} \indep_{X_{1:t}} Y_{1:t}.
	\end{equation}
	The set of all causal probability measures is denoted by $\mathcal{P}_c(\X \times \Y)$.
	Analogously, the set of anticausal measures $\mathcal{P}_{ac}(\X \times \Y)$ is defined as the set of measures $\pi \in \mathcal{P}(\X \times \Y)$ such that $\pi_{0, t+1}^{X_{1:t}, Y_{1:t}} = \pi_{0, t+1}^{X_0, Y_{1:t}}$ holds $\pi$ almost surely for all $t=1, \dots, N-1$. The set of bicausal measures $\mathcal{P}_{bc}(\X \times \Y)$ is defined as $\mathcal{P}_{bc}(\X \times \Y) := \mathcal{P}_c(\X \times \Y) \cap \mathcal{P}_{ac}(\X \times \Y)$. 
\end{definition}

Note that when $\mu$ is the $\X$-marginal of $\pi$, the right-hand side of \eqref{def:causal_property} coincides in fact $\pi$-a.s.\ with $\mu_{t+1}^{X_{1:t}}$.
In particular, the kernel $\pi^{X_{1:t},Y_{1:t}}_{t+1,0}$ does not depend on $Y_{1:t}$.

\begin{remark}
	\label{rem:equivdefcausal}
	By the chain rule of conditional independence \cite[Proposition 5.8]{Ka97}, we have
	\[
	X_{t} \indep_{X_{1:t-1}} Y_{1:t-1} \text{ and } X_{t+1} \indep_{X_{1:t}} Y_{1:t-1} \iff X_{t:t+1} \indep_{X_{1:t-1}} Y_{1:t-1}.
	\]
	Thus, this reasoning yields that $\pi \in \mathcal P_c(\X \times \Y)$ if and only if, for $t \in \{1,\ldots, N-1\}$,
	\begin{equation}
	\label{def:causal_property_random_variables_entire_path}
	X \indep_{X_{1:t}} Y_{1:t} \quad \text{under } \pi,
	\end{equation}
	 or equivalently, we have $\pi$-almost surely
	\begin{equation}
	\label{def:causal_property_kernel_entire_path}
	\pi_{0,t}^{X_{1:N},Y_0} = \pi_{0,t}^{X_{1:t},Y_0}.
	\end{equation}
\end{remark}

Let $c : \X \times \Y \rightarrow \mathbb{R}$ be continuous.
Define the set of causal, anticausal and bicausal couplings between two measures $\mu \in \mathcal{P}(\X)$ and $\nu \in \mathcal{P}(\Y)$ as $\Pi_c(\mu, \nu) := \Pi(\mu, \nu) \cap \mathcal{P}_c(\X \times \Y)$, $\Pi_{ac}(\mu, \nu) := \Pi(\mu, \nu) \cap \mathcal{P}_{ac}(\X \times \Y)$ and $\Pi_{bc}(\mu, \nu) := \Pi(\mu, \nu) \cap \mathcal{P}_{bc}(\X \times \Y)$, respectively. Whenever $p \in [1, \infty)$ is specified, $c$ is assumed to have growth of order at most $p$ and $\mu \in \mathcal{P}_p(\X), \nu \in \mathcal{P}_p(\X)$.

Whenever a statement is the same for the regular, causal, anticausal or bicausal version, we will use an index $\bullet$ as a placeholder. 
For example, $\Pi_\bullet(\mu, \nu)$ can mean either of $\Pi(\mu, \nu)$, $\Pi_c(\mu, \nu)$, $\Pi_{ac}(\mu, \nu)$, $\Pi_{bc}(\mu, \nu)$.
We define the (adapted) optimal transport problem as 
\[
V_{\bullet}(\mu, \nu, c) := \inf_{\pi \in \Pi_{\bullet}(\mu, \nu)} \int c\,d\pi.
\]

We denote by $D_{\rm KL}(\pi, \tilde{\pi})$ the relative entropy between $\pi, \tilde{\pi} \in \mathcal{P}(\X \times \Y)$ given by $D_{\rm KL}(\pi, \tilde{\pi}) = \int \log(\frac{d\pi}{d\tilde{\pi}}) \,d\pi$ if $\pi \ll \tilde{\pi}$ (i.e., if $\pi$ is absolutely continuous w.r.t.~$\tilde{\pi}$) and $D_{\rm KL}(\pi, \tilde{\pi}) = \infty$, else.
For $\varepsilon > 0$, we define the entropic versions of the (adapted) optimal transport problem by
\[
E^{\varepsilon}_{\bullet}(\mu, \nu, c) := \inf_{\pi \in \Pi_{\bullet}(\mu, \nu)} \int c\,d\pi + \varepsilon D_{\rm KL}(\pi, P).
\]
If $\varepsilon = 1$, we will simply write $E_{\bullet}(\mu, \nu, c)$.
\begin{remark}[Existence] \label{rem:existence_optimizers}
We remark that, when the cost $c$ is lower semicontinuous and lower-bounded, existence of optimizers to the problems $E_\bullet^\varepsilon$ and $V_\bullet$ follows by standard arguments due to weak compactness of $\Pi_\bullet(\mu,\nu)$, c.f.\ Lemma \ref{lem:all sets of couplings are compact}.
Moreover, by strict convexity of the relative entropy the optimizer of $E_\bullet^\varepsilon$ is unique.
\end{remark}

\section{Stability}
\label{sec:stab}
For $\mu, \tilde{\mu} \in \mathcal{P}_p(\X)$ and $p\in [1, \infty)$ we define the (adapted) versions of the Wasserstein distance $W_{p, \bullet}$ by
\[
W_{p,\bullet}(\mu, \tilde{\mu})^p := \inf_{\pi \in \Pi_{\bullet}(\mu, \tilde{\mu})} 
\int d_{\X}^p \,d\pi,
\]
and similarly on $\Y$. For $\pi, \tilde{\pi} \in \mathcal{P}(\X \times \Y)$, the (adapted) Wasserstein distance is defined by
\[
W_{p, \bullet}(\pi, \tilde{\pi})^p := \inf_{\theta \in \Pi_{\bullet}(\pi, 
	\tilde{\pi})} \int d_{\X}(x, \tilde{x})^p + d_{\Y}(y, \tilde{y})^p \,\theta(dx, dy, 
d\tilde{x}, d\tilde{y}),
\]
For the definition of causality for $\theta \in \mathcal{P}((\X \times \Y) \times (\X \times \Y))$, we regard the spaces $(\X \times \Y)$ as a single space, and the usual definition, that is Definition \ref{def:causality}, applies. For disintegrations of $\theta \in \mathcal{P}((\X \times \Y) \times (\X \times \Y))$ we use the notation $\theta_{A, B, C, D}$ for marginals, $\theta^{X_A, Y_B, \tilde{X}_A, \tilde{Y}_B}_{A_2, B_2, C_2, D_2}$ for disintegrations, etc. For instance, $\theta$ being causal means that 
\[
\theta_{t+1, t+1, 0, 0}^{X_{1:t}, Y_{1:t}, \tilde{X}_{1:t}, \tilde{Y}_{1:t}} = \theta_{t+1, t+1, 0, 0}^{X_{1:t}, Y_{1:t}, \tilde{X}_0, \tilde{Y}_0},
\]
for all $t=1, \dots, N-1$.


\subsection{Adapted version of the shadow}
\label{subsec:shadow}
The convolution of two stochastic kernels $K: \X \to \mathcal{P}(Y)$ and $\tilde{K}: \Y \to \mathcal{P}(\Z)$ is given by the (unique) kernel $K \ast \tilde K \colon \X \to \mathcal{P}(Z)$ satisfying
\[
\int f(z) \, K \ast \tilde K(x,dz) = \int \int f(z) \, \tilde K(y,dz) \, K(x,dy),
\]
for $f \in C_b(\Z)$ and $x \in \X$.

Concatenation of different couplings will be a recurrent tool throughout this paper. The following Lemma establishes the key property that concatenation of causal couplings remains causal.
\begin{lemma}
	\label{lem:concat}
	Let $\mu \in \mathcal{P}(\X), \nu \in \mathcal{P}(\Y), \gamma \in \mathcal{P}(\Z)$.
	Let $\pi \in \Pi_c(\mu, \nu), \tilde{\pi} \in \Pi_c(\nu, \gamma)$ where we write $\pi = \mu \otimes K$ and $\tilde{\pi} = \nu \otimes \tilde{K}$.
	Then
	\begin{equation}
	\label{eq:concat_well_defined}
	\mu \otimes (K \ast \tilde K) \in \Pi_c(\mu, \gamma).
	\end{equation}
	\begin{proof}
		Let $\theta := \mu \otimes (K * \tilde{K})$.
		Let $f \in C_b(\X \times \Z)$ and note that by Fubini's theorem
		\begin{align*}
		\int f \, d\theta
		&=
		\int \int \int f(x,z) \, \tilde K(y,dz) \, K(x,dy) \, \mu(dx) \\
		&=
		\int \int f(x,z) \, \tilde K(y,dz) \, \pi(dx,dy).
		\end{align*}
		Similarly to the above calculation, it is straightforward to see that $\theta \in \Pi(\mu,\gamma)$.
		
		It remains to show that $\theta \in \Pi_c(\mu,\gamma)$.
		We write $\hat \gamma := \mu \otimes K \otimes \tilde K$.
		Since $\theta = (X,Z)_\ast \hat \gamma$,  the causality of $\theta$ is equivalent to the following:
		For $t \in \{1, \dots, N-1\}$, we have
		\begin{equation}
		\label{eq:hat_gamma_causality_XZ}
		X \indep_{X_{1:t}} Z_{1:t} \quad \text{under }\hat \gamma.
		\end{equation}
		By definition of $\hat \gamma$, we have $\pi$-almost surely that $\hat \gamma^{x,y}(dz) = \tilde K(y,dz)$, thus,
		\begin{equation}
		\label{eq:hat_gamma_CI_product}
		X \indep_Y Z\quad \text{under }\hat \gamma.
		\end{equation}
		Let $t \in \{1,\ldots, N-1\}$.
		Since $(Y,Z)_\ast \hat \gamma = \tilde \pi$ we have, by causality of the latter,
		\begin{equation}
		\label{eq:hat_gamma_causality_YZ}
		Y \indep_{Y_{1:t}} Z_{1:t} \quad \text{under }\hat \gamma.
		\end{equation}
		Using \eqref{eq:hat_gamma_CI_product} and \eqref{eq:hat_gamma_causality_YZ} we find by the chain rule \cite[Proposition 5.8]{Ka97} that
		\begin{equation}
		\label{eq:hat_gamma_causality_XZ_tmp}
		(X, Y) \indep_{Y_{1:t}} Z_{1:t} \text{ and thus } X \indep_{(X_{1:t}, Y_{1:t})} Z_{1:t} \quad \text{under }\hat \gamma.
		\end{equation}
		Similarly, we have due to $(X,Y)_\ast \hat \gamma = \pi$ and $\pi \in \mathcal P_c(\X \times \Y)$ that
		\begin{equation}
		\label{eq:hat_gamma_causality_XY}
		X \indep_{X_{1:t}} Y_{1:t} \quad \text{under }\hat\gamma.
		\end{equation}
		Finally, we find that \eqref{eq:hat_gamma_causality_XZ_tmp} and \eqref{eq:hat_gamma_causality_XY} imply \eqref{eq:hat_gamma_causality_XZ} by the chain rule \cite[Proposition 5.8]{Ka97}, which concludes the proof.
	\end{proof}

\end{lemma}

Using the previous Lemma together with Minkowski's inequality immediately yields the following:
\begin{corollary}[To Lemma \ref{lem:concat}]
	\label{cor:triangle}
	Let $\mu, \nu, \gamma \in \mathcal{P}(\X)$. Then $W_{p, c}(\mu, \gamma) \leq W_{p, c}(\mu, \nu) + W_{p, c}(\nu, \gamma)$.
\end{corollary}

\begin{definition}[Shadow]
	\label{def:shadow}
	Let $\mu, \tilde\mu \in \mathcal{P}(\X), \nu, \tilde{\nu} \in \mathcal{P}(\Y)$.
	Let $\pi \in \Pi(\mu, \nu)$, $\gamma \in \Pi(\mu, \tilde{\mu})$ and $\gamma' \in \Pi(\nu, \tilde{\nu})$ where we write $\gamma = \mu \otimes K^{\mu}$ and $\gamma' = \nu \otimes K^\nu$.
	Define 
	\[ K\colon \X \times \Y \to \mathcal P(\X \times \Y) \colon (x,y) \mapsto K^\mu(x) \otimes K^\nu(y). \] 
	Then we call $\tilde{\pi} := \pi K \in \Pi(\tilde{\mu}, \tilde{\nu})$, c.f.\ \eqref{eq:def.muK}, the \textsl{$(K^\mu, K^\nu)$-shadow} of $\pi$.
\end{definition}

\begin{lemma}[Properties of the shadow]
	\label{lem:prop_shadow} Let $p \in [1, \infty)$.
	Let $\mu, \tilde\mu \in \mathcal{P}_p(\X), \nu, \tilde{\nu} \in \mathcal{P}_p(\Y)$, $\pi \in \Pi(\mu, \nu)$, $\gamma \in \Pi(\mu, \tilde{\mu})$ and $\gamma' \in \Pi(\nu, \tilde{\nu})$ where we write $\gamma = \mu \otimes K^{\mu}$ and $\gamma' = \nu \otimes K^\nu$.
	Let $\tilde{\pi}$ be the $(K^\mu, K^\nu)$-shadow of $\pi$  and denote by $C_{K^\mu} := \int d_\X(x, \tilde{x})^p \,\mu \otimes K^\mu(dx, d\tilde{x})$ the transport cost between $\mu$ and $\tilde{\mu}$ associated with the coupling $K^\mu$, and similarly $C_{K^\nu}$.
	\begin{itemize}
		\item[(i)] It holds
		\[
		\int\int d_{\X}^p(x, \tilde{x})+d_{\Y}^p(y, \tilde{y}) \,K(x, y)(d\tilde{x}, d\tilde{y}) \, \pi(dx, dy) = C_{K^\mu} + C_{K^\nu}.
		\]
		\item[(ii)] For any other coupling $\kappa \in \Pi(\mu, \nu)$ and its $(K^\mu, K^\nu)$-shadow $\tilde{\kappa}$ it holds
		\[
		D_{\rm KL}(\tilde{\pi}, \tilde\kappa) \leq D_{\rm KL}(\pi, \kappa).
		\]
		In particular, since $\tilde\mu \otimes \tilde{\nu}$ is the $(K^\mu, K^\nu)$-shadow of $P$, it holds $D_{\rm KL}(\tilde{\pi}, \tilde{\mu}\otimes\tilde{\nu}) \leq D_{\rm KL}(\pi, P)$
		\item[(iii)] If $\pi$ is causal, $\mu \otimes K^\mu$ is anticausal and $\nu \otimes K^\nu$ is causal, then $\tilde{\pi}$ is causal. If $\mu \otimes K^\mu$ and $\nu \otimes K^\nu$ are $W_{p,ac}(\mu, \tilde{\mu})$- and $W_{p,c}(\nu, \tilde{\nu})$-optimal, then $W_p(\pi, \tilde{\pi})^p \leq W_{p,ac}(\mu, \tilde{\mu})^p + W_{p,c}(\nu, \tilde{\nu})^p.$
		\item[(iv)] If $\pi$ is anticausal, $\mu \otimes K^\mu$ is causal and $\nu \otimes K^\nu$ is anticausal, then $\tilde{\pi}$ is anticausal. If $\mu \otimes K^\mu$ and $\nu \otimes K^\nu$ are $W_{p, c}(\mu, \tilde{\mu})$- and $W_{p, ac}(\nu, \tilde{\nu})$-optimal, then $W_p(\pi, \tilde{\pi})^p \leq W_{p, c}(\mu, \tilde{\mu})^p + W_{p, ac}(\nu, \tilde{\nu})^p.$
		\item[(v)] Let $\theta = \pi \otimes K \in \Pi(\pi, \tilde{\pi})$. If $\mu \otimes K^\mu$ and $\nu \otimes K^\nu$ are causal, then $\theta \in \Pi_c(\pi, \tilde{\pi})$. If $\mu \otimes K^\mu$ and $\nu \otimes K^\nu$ are $W_{p, c}(\mu, \tilde{\mu})$- and $W_{p, c}(\nu, \tilde{\nu})$-optimal, then $W_{p, c}(\pi, \tilde{\pi})^p \leq W_{p, c}(\mu, \tilde{\mu})^p + W_{p, c}(\nu, \tilde{\nu})^p.$
	\end{itemize}
	\begin{proof}
		The first property (i) is clear by definition of $K$. Property (ii) follows immediately from the data processing inequality (cf.~\cite[Lemma 4.1]{eckstein2021quantitative}).
		
		(iii), (iv) are obviously equivalent and we show only (iii).
		 Write $K^{\tilde \mu}$ for the disintegration kernel of $\gamma$ w.r.t.\ the second coordinate.
		Since $\gamma = \mu \otimes K^\mu$ is anticausal, we find that $\tilde \mu \otimes K^{\tilde \mu}$ is causal. 
		Write $\pi = \mu \otimes S$. Then, we find that
		\[
		\theta = \pi \otimes K = (\mu \otimes S) \otimes K^\mu \otimes K^\nu = (\mu \otimes K^\mu) \otimes S \otimes K^\nu = \tilde{\mu} \otimes K^{\tilde{\mu}} \otimes S \otimes K^\nu,
		\]
		That is, $\theta$ can be defined as a concatenation of causal kernels in such an order that $\tilde{\pi}$ is the joint first and last marginal. 
		 I.e., $\tilde{\pi} = \tilde{\mu} \otimes (K^{\tilde{\mu}} \ast S \ast K^\nu)$ and by Lemma \ref{lem:concat} it thus follows that $\tilde{\pi} \in \Pi_c(\tilde\mu, \tilde{\nu})$. The second claim in (iii) follows from (i).
		
		(v): We have to show that $\theta_{t+1, t+1, 0, 
			0}^{X_{1:t}, Y_{1:t}, \tilde{X}_{1:t}, \tilde{Y}_{1:t}} = 
		\theta_{t+1, t+1, 0, 0}^{X_{1:t}, Y_{1:t}, 
			\tilde{X}_0, \tilde{Y}_0}$ for all $t=1, \dots, N$. Let us 
		denote by $\alpha = \mu\otimes K^\mu$, where $K^\mu = \alpha_{0, 
			1:N}^{X_{1:N}, \tilde{X}_0}$ and note that by causality 
		$\alpha_{0, t}^{X_{1:N}, \tilde{X}_0} = \alpha_{0, 
			t}^{X_{1:t}, \tilde{X}_0}$, c.f.\ Remark \ref{rem:equivdefcausal}, and the same for $\beta = \nu \otimes K^\nu$. 
		Thus $\theta$ can be disintegrated as
		\[
		\theta = \pi \otimes K = \pi_{1:t, 1:t} \otimes \alpha_{0, 1:t}^{X_{1:t}, \tilde{X}_0} \otimes \beta_{0, 1:t}^{Y_{1:t}, \tilde{Y}_0} \otimes \pi_{t+1, t+1}^{X_{1:t}, Y_{1:t}} \otimes (...)
		\]
		where the terms in (...) are irrelevant for $\theta_{1:t+1, 1:t+1, 1:t, 1:t}$ and thus
		\[
		\theta_{1:t+1, 1:t+1, 1:t, 1:t} = \pi_{1:t, 1:t} \otimes \alpha_{0, 1:t}^{X_{1:t}, \tilde{X}_0} \otimes \beta_{0, 1:t}^{Y_{1:t}, \tilde{Y}_0} \otimes \pi_{t+1, t+1}^{X_{1:t}, Y_{1:t}} = \theta_{1:t, 1:t, 1:t, 1:t} \otimes \pi_{t+1, t+1}^{X_{1:t}, Y_{1:t}},
		\]
		which implies $\theta_{t+1, t+1, 0, 0}^{X_{1:t}, Y_{1:t}, \tilde{X}_{1:t}, \tilde{Y}_{1:t}} = \pi_{t+1, t+1}^{X_{1:t}, Y_{1:t}}$, which shows the claim.
	\end{proof}
\end{lemma}

\subsection{Stability of adapted optimal transport and its entropic versions}
\label{subsec:stab}
For stability of the entropic problems (in particular, for the convergence of the causal Sinkhorn algorithm later), it will be necessary to control adapted distances through relative entropy. We will even show that adapted distances can -- for bounded settings -- be controlled in a quantitative fashion via the total variation distance.
For two measures $\mu, \tilde{\mu} \in \mathcal{P}(\X)$, we recall the lattice minimum $\mu \land \tilde{\mu}(A) := \inf_{B \subset A \text{ Borel}} \mu(B) + \tilde{\mu}(A\backslash B)$ for Borel sets $A$, or equivalently one can define $\mu \land \tilde{\mu}$ via its density
\[
\frac{d \mu \land \tilde{\mu}}{d \mu + \tilde{\mu}} := \min\left\{\frac{d\mu}{d\mu + \tilde{\mu}}, \frac{d\tilde\mu}{d\mu + \tilde{\mu}}\right\}.
\]
For reference, we recall that the following equality holds
\begin{align}
\label{def:TV} \tag{TV}
\|\mu - \tilde{\mu}\|_{TV} :=&  \sup_{B \subseteq \X \text{ Borel}} \mu(B) - \tilde \mu(B)
\\
\nonumber 
=& 1 - (\mu \land \tilde{\mu}(\X)) = \inf_{\pi \in \Pi(\mu, \tilde{\mu})} \int \eins_{x \neq \tilde{x}} \,\pi(dx, d\tilde{x}),
\end{align}
see, e.g., \cite[Definition 2.15]{Massart.07} and \cite[Lemma 2.20]{Massart.07}.
Further, we define 
\begin{equation}
\label{def:AV}
\tag{AV}
\|\mu - \tilde{\mu}\|_{AV} := \inf_{\pi \in \Pi_{bc}(\mu, \tilde{\mu})} \int \eins_{x \neq \tilde{x}} \,\pi(dx, d\tilde{x}),
\end{equation}
which we call the adapted variation.
The following results establishes that total variation and adapted variations are equivalent up to a constant.

\begin{lemma}
	\label{lem:AV_TV}
	Let $\mu, \tilde{\mu} \in \mathcal P(\X)$.
	Then we have
	\begin{equation}
	\label{eq:AV_TV_statement}
	\|\mu - \tilde{\mu} \|_{TV} \leq \| \mu - \tilde{\mu} \|_{AV} \le \left( 2^N - 1 \right) \| \mu - \tilde{\mu} \|_{TV}.
	\end{equation}
\end{lemma}

\begin{proof}
	The assertion is trivially satisfied when $N = 1$, since then adapted variation and total variation coincide.
	We show the statement by induction:
	Assume that \eqref{eq:AV_TV_statement} holds for $N \in \mathbb N$.
	Let $\X$ be a path space with $(N+1)$-time steps and $\mu,\tilde{\mu} \in \mathcal P(\X)$, then
	\[
	\| \mu_{1:N} - \tilde{\mu}_{1:N} \|_{AV} \le \left( 2^N - 1 \right)
	\| \mu_{1:N} - \tilde{\mu}_{1:N} \|_{TV} \le \left( 2^N - 1 \right) 
	\| \mu - \tilde{\mu} \|_{TV},
	\]
	where the last inequality holds by the data processing inequality for the total variation, see for example \cite[Lemma 4.1]{eckstein2021quantitative} and note that the proof therein also works for the total variation distance.
	Let $\pi^N \in \Pi_{bc}(\mu_{1:N}, \tilde{\mu}_{1:N})$ and write $\eta^{x_{1:N}} := \mu_{N+1}^{x_{1:N}} \land \tilde{\mu}_{N+1}^{x_{1:N}}$ for $x_{1:N} \in \X_{1:N}$.
	We can define a bicausal coupling $\pi \in \Pi_{bc}(\mu,\tilde{\mu})$ by
	\begin{align*}
	\pi 
	&:= 
	\pi^N \otimes \pi_{N+1,N+1}^{x_{1:N}, y_{1:N}},
	\quad
	\text{where }
	\\
	\pi_{N+1,N+1}^{x_{1:N}, y_{1:N}} 
	&:=
	\begin{cases}
	(\id,\id)_\ast \eta^{x_{1:N}} 
	+ 
	(\mu_{N+1}^{x_{1:N}} - \eta^{x_{1:N}}) 
	\otimes
	(\tilde{\mu}_{N+1}^{x_{1:N}} - \eta^{x_{1:N}})
	&
	x_{1:N} = y_{1:N},
	\\
	\mu_{N+1}^{x_{1:N}}
	\otimes
	\tilde{\mu}_{N+1}^{x_{1:N}}
	&
	\text{else}.
	\end{cases}
	\end{align*}
	This allows us to estimate the adapted variation:
    Consider the diagonal in $\X_{1:N} \times \Y_{1:N}$ that is the set of pairs $(x_{1:N},y_{1:N}) \in \X_{1:N} \times \Y_{1:N}$ with $x_{1:N} = y_{1:N}$ and write $\eta := (X_{1:N})_\ast (\eins_{X_{1:N} = Y_{1:N}} \pi^N)$, where $\eins_{X_{1:N} = Y_{1:N}} \pi^N$ is the measure which has density $\eins_{X_{1:N} = Y_{1:N}}$ w.r.t.~$\pi^N$.
    Note that $\eta$ essentially corresponds to the part of $\pi^N$ that is supported on the diagonal.
    We find
	\begin{align}
	\nonumber
	\| \mu - \tilde{\mu} \|_{AV} 
	&
	\le \int \eins_{x \neq y} \, \pi(dx,dy)
	\\
	\nonumber
	&
	= \int \eins_{x_{1:N} \neq y_{1:N}} \, \pi^N(dx_{1:N}, dy_{1:N})
	+
	\int \left(1 - \eta^{x_{1:N}}(\X_{N+1}) \right) \, \eta(dx_{1:N})
	\\
	\nonumber
	&
	= \int \eins_{x_{1:N} \neq y_{1:N}} \, \pi^N(dx_{1:N}, dy_{1:N})
	+
	\int \| \mu^{x_{1:N}}_{N+1} - \tilde{\mu}^{x_{1:N}}_{N+1} \|_{TV} \, \eta(dx_{1:N})
	\\
	\label{eq:AV estimate step 1}
	&
	= \int \eins_{x_{1:N} \neq y_{1:N}} \, \pi^N(dx_{1:N}, dy_{1:N})
	+
	\| \eta \otimes \left( \mu^{x_{1:N}}_{N+1} - \tilde{\mu}^{x_{1:N}}_{N+1} \right) \|_{TV},
	\end{align}
	 which can be seen as follows: 
	the inequality is due to $\pi \in \Pi_{bc}(\mu,\tilde \mu)$ and the definition \eqref{def:AV}.
	The first equality follows from decomposing $\pi^N$ into two measures $\eta$ and $\eins_{x_{1:n} \neq y_{1:n}} \pi^N$, and then using the definition of the kernel $\pi^{x_{1:N},y_{1:N}}_{N+1,N+1}$.
	The second inequality stems from \eqref{def:TV} and the definition of the kernel $\eta^{x_{1:N}}$.
	Finally, the last equality is a straight-forward application of Fubini's Theorem.
	
	Let $G \subseteq \X$ be measurable and write $G^{x_{1:N}} := \{ x_{N+1} \colon (x_{1:N}, x_{N+1}) \in G \}$.
	We compute
	\begin{align}
	\nonumber
	\eta \otimes \left( \mu^{x_{1:N}}_{N+1} - \tilde{\mu}^{x_{1:N}}_{N+1} \right) (G)
	&=
	\int \mu^{x_{1:N}}_{N+1}(G^{x_{1:N}}) - \tilde{\mu}^{x_{1:N}}_{N+1}(G^{x_{1:N}}) \, \eta(dx_{1:N})
	\\
	\nonumber
	&=
	\mu(G) - \tilde{\mu}(G) 
	+ \int \mu^{x_{1:N}}_{N+1}(G^{x_{1:N}}) \, (\eta - \mu_{1:N})(dx_{1:N})
	\\
	\nonumber
	&\phantom{=\mu(G)}
	- \int \tilde{\mu}^{x_{1:N}}_{N+1}(G^{x_{1:N}}) \, (\eta - \tilde{\mu}_{1:N})(dx_{1:N})
	\\
	\nonumber
	&\le
	\mu(G) - \tilde{\mu}(G) - \int 1 \, (\eta - \tilde{\mu}_{1:N})(dx_{1:N})
	\\
	\label{eq:AV estimate step 2}
	&\le
	\| \mu - \tilde{\mu} \|_{TV} + \int \eins_{x_{1:N} \neq y_{1:N}}  \pi^N(dx_{1:N}, dy_{1:N}).
	\end{align}
	 Note that the first inequality is due to $\eta \le \mu_{1:N}$ and the second inequality follows from \eqref{def:TV}.
	By plugging \eqref{eq:AV estimate step 2} into \eqref{eq:AV estimate step 1}, we obtain, as $G$ and $\pi^N$ were arbitrary,
	\begin{align*}
	\| \mu - \tilde{\mu} \|_{AV}
	&\le
	2 \| \mu_{1:N} - \tilde{\mu}_{1:N} \|_{AV}
	+
	\| \mu - \tilde{\mu} \|_{TV}
	\\
	&\le
	\left( 2^{N+1} - 2 + 1 \right) \| \mu - \tilde{\mu} \|_{TV}
	=
	\left( 2^{N+1} - 1 \right) \| \mu - \tilde{\mu} \|_{TV},
	\end{align*}
	where the last inequality is due to the induction hypothesis.
\end{proof}

Recall that by Remark \ref{rem:existence_optimizers} there exist optimizers of the problems $V_\bullet$ and $E_\bullet ^\varepsilon$ as long as the cost is lower semicontinuous and lower-bounded.
We can state the first main result:
\begin{theorem}[Qualitative stability]
	\label{thm:qual_stab}
	Let $p \in [1, \infty)$ and $c$ be continuous and have growth of order at most $p$.
	Let $\mu^n \in \mathcal{P}_p(\X)$ and $\nu^n \in \mathcal{P}_p(\Y)$, $n \in \mathbb{N}$, be a sequences converging to
	$\mu \in \mathcal P_p(\X)$ resp.\ $\nu \in \mathcal{P}_p(\Y)$ in $W_{p,bc}$.
	Then
	\[
	\lim_{n\rightarrow \infty} V_{\bullet}(\mu^n, \nu^n, c) = V_\bullet(\mu, \nu, c),
	\]
	and accumulation points of optimizers of $(V_{\bullet}(\mu^n, \nu^n, c))_{n\in\mathbb{N}}$ are optimizers of $V_\bullet(\mu, \nu, c)$. 
	Further,
	\[
	\lim_{n\rightarrow \infty} E^\varepsilon_{\bullet}(\mu^n, \nu^n, c) = E^\varepsilon_{\bullet}(\mu, \nu, c)
	\]
	and the associated optimizers converge in $W_p$.
	\begin{proof}
		In this proof, subsequences will be relabeled by the initial index.
		Since we will require shadows in this proof, we fix the notation $\mu^n \otimes K^{\mu, n} = \mu \otimes \tilde{K}^{\mu, n}$ as an optimizer of $W_{p, bc}(\mu^n, \mu)$ and similarly $\nu^n \otimes K^{\nu, n} = \nu \otimes \tilde{K}^{\nu, n}$ as an optimizer of $W_{p, bc}(\nu^n, \nu)$.
		
		We first show the statement for $V_\bullet$. Let $\pi^{*, n} \in \Pi(\mu^n, \nu^n)$ be optimizers of $V_\bullet(\mu^n, \nu^n, c)$ and $\pi^*$ be an optimizer of $V_\bullet(\mu, \nu, c)$.
		
		We show that any limit point $\pi$ of $(\pi^{*, n})_{n}$ satisfies $\pi \in \Pi_{\bullet}(\mu, \nu)$.
		To this end, let $\hat{\pi}^n$ be the $(K^{\mu, n}, K^{\nu, n})$-shadow of $\pi^{*, n}$. 
		By Lemma \ref{lem:prop_shadow} (i), (iii) and (iv), $\hat{\pi}^n \in \Pi_{\bullet}(\mu, \nu)$ and $\hat{\pi}^n \rightarrow \pi$ (in $W_p$) as well. Since $\Pi_{\bullet}(\mu, \nu)$ is closed by Lemma \ref{lem:all sets of couplings are compact}, $\pi\in \Pi_{\bullet}(\mu, \nu)$ follows.
		
		Since by weak compactness, any subsequence of $(\pi^{*, n})_{n\in\mathbb{N}}$ has a further subsequence converging to an accumulation point $\pi$, which by the above satisfies $\pi \in \Pi_{\bullet}(\mu, \nu)$, we get
		\[
		\liminf_{n\rightarrow \infty} V_\bullet(\mu^n, \nu^n, c) = \liminf_{n\rightarrow \infty} \int c \,d\pi^{*, n} \ge \inf_{\pi \in \Pi_{\bullet}(\mu,\nu} \int c \,d\pi \geq \int c \,d\pi^*.
		\]
		On the other hand, let $\pi^n$ be the $(\tilde{K}^{\mu, n}, \tilde{K}^{\nu, n})$-shadow of $\pi^*$. Then, by Lemma \ref{lem:prop_shadow} (i), (iii) and (iv), $\pi^n \in \Pi_{\bullet}(\mu^n, \nu^n)$ and $W_p(\pi^n, \pi^*) \rightarrow 0$. Thus
		\[
		\limsup_{n\rightarrow \infty} V_\bullet(\mu^n, \nu^n, c) \leq \lim_{n\rightarrow \infty} \int c \,d\pi^n = \int c \,d\pi^*,
		\]
		where the inequality follows from $\pi^n \in \Pi_{\bullet}(\mu^n, \nu^n)$ and the equality is due $W_p$-convergence of $\pi^n$ to $\pi^\ast$.
		
		Concerning the proof for $E^\varepsilon_\bullet$:
		The inequality $\liminf_{n\rightarrow \infty}E^\varepsilon_\bullet(\mu^n, \nu^n, c) \geq E^\varepsilon_\bullet(\mu, \nu, c)$ can be shown analogously as above, while using lower-semicontinuity of the relative entropy.
		On the other hand, the reverse inequality $\limsup_{n\rightarrow \infty} E^\varepsilon_\bullet(\mu^n, \nu^n, c) \leq E^\varepsilon_\bullet(\mu, \nu, c)$ again can be derived analogously by using the properties of the shadow, see Lemma \ref{lem:prop_shadow} (ii).
	\end{proof}
\end{theorem}

To quantify the respective stability results, we require some assumptions on both marginal distributions and the cost function:
\begin{definition}
	Let $p \in [1, \infty)$, $\mu \in \mathcal{P}(\X), \nu \in \mathcal{P}(\Y)$, and $L > 0$. We say that the cost function $c$ satisfies $A(L, p, \mu, \nu)$, if \[\tag{$A(L, p, \mu, \nu)$} \forall \pi, \tilde{\pi} \in \Pi(\mu, \nu):~ \Big|\int c \,d\pi - \int c \,d\tilde{\pi} \Big| \leq L\, W_p(\pi, \tilde{\pi}).\]
	Further, we say that $\mu, \nu$ satisfy $T(p, C_p)$, resp $T_a(p, C_p)$, for $C_p >0$, if 
	\begin{align}
	&\forall \pi, \tilde{\pi} \in \Pi(\mu, \nu):& ~ W_p(\pi, \tilde{\pi}) &\leq C_p D_{\rm KL}(\pi, \tilde{\pi})^{\frac{1}{2p}} \tag{$T(p, C_p)$} \\
	&\forall \pi, \tilde{\pi} \in \Pi(\mu, \nu):& ~ \max\{W_{p, c}(\pi, \tilde{\pi}), W_{p, ac}(\pi, \tilde{\pi})\} &\leq C_p D_{\rm KL}(\pi, \tilde{\pi})^{\frac{1}{2p}} \tag{$T_a(p, C_p)$}
	\end{align}
\end{definition}

Note that $A(L, p, \mu, \nu)$ is clearly a weaker assumption than Lipschitz continuity of $c$, and also satisfied in more general and relevant cases, see, e.g., \cite[Example 3.4 and Lemma 3.5]{eckstein2021quantitative}.
Criteria for $T(p, C_p)$ are given in \cite[Lemma 3.10]{eckstein2021quantitative}. The following establishes that $T_a(p, C_p)$ is satisfied in bounded spaces.
\begin{lemma}
	If $\diam(\X \times \Y) = \sup_{(x, y), (\tilde{x}, \tilde{y}) \in \X\times\Y}d_{\X}(x, \tilde{x}) +d_{\Y}(x, \tilde{y}) < \infty$, then any $\mu \in \mathcal{P}(\X)$ and $\nu \in \mathcal{P}(\Y)$ satisfy $T_a(p, C_p)$ with \[C_p = \diam(\X \times \Y) \,2^{\frac{N}{p}} \,2^{-\frac{1}{2p}}.\]
	\begin{proof}
		Combine Pinsker's inequality with Lemma \ref{lem:AV_TV}.
	\end{proof}
\end{lemma}

\begin{theorem}[Quantitative stability]
	\label{thm:quantitative_stab}
	Let $p \in [1, \infty)$, $\mu, \tilde{\mu} \in \mathcal{P}_p(\X), \nu, \tilde{\nu} \in \mathcal{P}_p(\Y)$, and $c$ satisfy $A(L, p, \mu, \nu)$ for $L>0$. Let $\Delta^p := W_{bc}(\mu, \tilde{\mu})^p + W_{bc}(\nu, \tilde{\nu})^p$.
	\begin{enumerate}[label = (\roman*)]
		\item \label{it:quantitative_stab_1}
		We have
		\begin{align*}|V_\bullet(\mu, \nu, c) - V_\bullet(\tilde{\mu}, \tilde{\nu}, c)| &\leq L\, \Delta,\\ |E_\bullet(\mu, \nu, c) - E_\bullet(\tilde{\mu}, \tilde{\nu}, c)| &\leq L\, \Delta.\end{align*}
		\item \label{it:quantitative_stab_2}
		Let $1\leq q < p$. Let $\pi^*$ be the optimizer of $E_\bullet(\mu, \nu, c)$ and $\tilde{\pi}^*$ be the optimizer of $E_\bullet(\tilde\mu, \tilde{\nu}, c)$. If $\mu, \nu$ satisfy $T(q, C_q)$, then
		\[
		W_q(\pi^*, \tilde{\pi}^*) \leq 2^{\frac{1}{q}-\frac{1}{p}} \Delta + C_q(2 L \Delta)^{\frac{1}{2q}}.
		\]
		If $\mu, \nu$ and $\tilde{\mu}, \tilde{\nu}$ even satisfy $T_a(q, C_q)$, then
		\[
		\max\{W_{q, c}(\pi^*, \tilde{\pi}^*), W_{q, ac}(\pi^*, \tilde{\pi}^*)\} \leq 2^{\frac{1}{q}-\frac{1}{p}} \Delta + C_q(2 L \Delta)^{\frac{1}{2q}}.
		\]
	\end{enumerate}
	\begin{proof}
		We fix the notation $\mu \otimes K^{\mu} = \tilde{\mu} \otimes \tilde{K}^{\mu}$ as an optimizer of $W_{p, bc}(\mu, \tilde{\mu})$ and $\nu \otimes K^\nu = \tilde{\nu} \otimes \tilde{K}^\nu$ as an optimizer of $W_{p, bc}(\nu, \tilde{\nu})$.
		
		\ref{it:quantitative_stab_1}:
		As the statements for $E_\bullet$ and $V_\bullet$ follow by similar reasoning, we only show them for $E_\bullet$:
		Let $\tilde{\pi}^*$ be an optimizer of $E_\bullet(\tilde{\mu}, \tilde{\nu}, c)$ and $\pi$ be the $(\tilde{K}^{\mu}, \tilde{K}^\nu)$-shadow of $\tilde{\pi}^*$. Then, since $\pi \in \Pi_{\bullet}(\mu, \nu)$ (cf.~Lemma \ref{lem:prop_shadow}(iii) and (iv)), and using Lemma \ref{lem:prop_shadow}(i) and (ii), as well as the assumption on $c$, we get
		\begin{align*}
		E_\bullet(\mu, \nu, c) &\leq \int c\,d\pi + D_{\rm KL}(\pi, P) \\
		&\leq \int c \,d\tilde{\pi}^* + D_{\rm KL}(\tilde{\pi}^*, \tilde{\mu} \otimes \tilde{\nu}) + \big|\int c \,d(\pi - \tilde{\pi})\big| \\
		&\leq E_\bullet(\tilde\mu, \tilde{\nu}, c) + L\, W_p(\pi, \tilde{\pi}) \\
		&\leq E_\bullet(\tilde\mu, \tilde{\nu}, c) + L\, \Delta.
		\end{align*}
		The reverse inequality follows by symmetry, and thus the proof of (i) is completed.
		
		\ref{it:quantitative_stab_2}:
		We introduce additional notation for the values
		\[
		\mathcal{F}(\pi) := \int c\,d\pi + D_{\rm Kl}(\pi, P) 
		\text{ and }
		\tilde{\mathcal{F}}(\tilde\pi) := \int c\,d\tilde{\pi} + D_{\rm KL}(\tilde\pi, \tilde{\mu}\otimes \tilde{\nu}),
		\]
		where $\pi, \tilde \pi \in \mathcal P(\X \times \Y)$.
		First, we treat the case where $\mu, \nu$ satisfy $T(q, C_q)$.
		
		Let $\pi$ be the $(\tilde{K}^{\mu}, \tilde{K}^\nu)$-shadow of $\tilde{\pi}^*$. The basic method of proof is the following:
		\begin{align*}
		W_q(\pi^*, \tilde{\pi}^*) &\leq W_q(\tilde{\pi}^*, \pi) + W_q(\pi, \pi^*) \\
		&\leq 2^{\frac{1}{q}-\frac{1}{p}} \,W_p(\tilde{\pi}^*, \pi) + C_q\,D_{\rm KL}(\pi, \pi^*)^{\frac{1}{2q}} \\
		&\stackrel{(1)}{\leq} 2^{\frac{1}{q}-\frac{1}{p}} \,\Delta + C_q\,\big(\mathcal{F}(\pi) - \mathcal{F}(\pi^*)\big)^{\frac{1}{2q}} \\
		&\stackrel{(2)}{\leq} 2^{\frac{1}{q}-\frac{1}{p}} \,\Delta + C_q\,\big(2 \, L \,\Delta\big)^{\frac{1}{2q}}.
		\end{align*}
		While the second inequality is just assumption $T(q, C_q)$ for $\mu$ and $\nu$, it remains to justify steps (1) and (2). Step (1) follows by the Pythagorean theorem for relative entropy \cite[Theorem 2.2]{csiszar1975divergence}, combined with the standard reformulation for entropic optimal transport (cf.~\cite[Lemma 4.4]{eckstein2021quantitative}, which works entirely analogously for the adapted optimal transport). Step (2) follows by noting that the proof for $(i)$ actually shows $\mathcal{F}(\pi) \leq \tilde{\mathcal{F}}(\tilde{\pi}^*) + L \,\Delta$ and thus $\mathcal{F}(\pi) \leq \mathcal{F}(\pi^*) + 2\,L\,\Delta$.
		
		Next, we assume that $\mu, \nu$ satisfy $T_a(q, C_q)$. Then, the basic method of proof works as above, where we additionally use that by Lemma \ref{lem:prop_shadow}(v), it holds $W_{p, c}(\tilde\pi^*, \pi) \leq \Delta$, since the coupling induced by the shadow is causal. Thus, we get
		\begin{align*}
		W_{q, c}(\pi^*, \tilde{\pi}^*) &\leq W_{q, c}(\tilde{\pi}^*, \pi) + W_{q, c}(\pi, \tilde{\pi}^*) \\
		&\leq 2^{\frac{1}{q}-\frac{1}{p}} W_{p, c}(\tilde{\pi}^*, \pi) + C_q\,D_{\rm KL}(\pi, \pi^*)^{\frac{1}{2q}} \\
		&\leq 2^{\frac{1}{q}-\frac{1}{p}} \Delta + C_q\,\big(\mathcal{F}(\pi) - \mathcal{F}(\pi^*)\big)^{\frac{1}{2q}} \\
		&\leq 2^{\frac{1}{q}-\frac{1}{p}} \Delta
		+ C_q\,\big(2 \, L \,\Delta\big)^{\frac{1}{2q}}.
		\end{align*}
		The anticausal direction works analogously by using assumption $T_a(p, C_q)$ for $\tilde{\mu}, \tilde{\nu}$. This completes the proof.
	\end{proof}
\end{theorem}

\begin{remark}
	We shortly mention that in the above Theorems \ref{thm:qual_stab} and \ref{thm:quantitative_stab}, some results are stated in less generality than possible to avoid overly lengthy statements.
	
	For instance in Theorem \ref{thm:qual_stab}, $\limsup_{n\rightarrow \infty} V_{c}(\mu^n, \nu^n, c) \leq V_{c}(\mu, \nu, c)$ holds already if only $W_{p, c}(\mu^n, \mu) \rightarrow 0$ and $W_{p, ac}(\nu^n, \nu) \rightarrow 0$, by Lemma \ref{lem:prop_shadow} (iii).
	Parts of the conclusions can also be strengthened in Theorem \ref{thm:qual_stab}. For instance, we can conclude that optimal couplings converge with respect to $\max\{W_{p, c}, W_{p, ac}\}$.
	
	In Theorem \ref{thm:quantitative_stab}, similar specifications can apply. For instance, in part (ii), if only $\mu, \nu$ satisfy $T_a(q, C_q)$ (but not $\tilde\mu, \tilde\nu$), we can still obtain the given bound for $W_{q, c}(\pi^*, \tilde{\pi}^*)$ (but no longer for $W_{q, ac}(\pi^*, \tilde{\pi}^*)$).
\end{remark}

\subsection{Discretization and LP formulation}
\label{subsec:discretization}
The baseline approach for solving infinite-dimensional optimization problems like optimal transport is via approximation by a discrete setting, thereby making the optimization problem finite-dimensional. In optimal transport, this translates to discretizing the marginal distributions. The stability results obtained in Section \ref{subsec:stab} show that the deviation of both optimal values and optimizers can be controlled when discretizing the marginals suitably, i.e., in a sense that they approximate the initial marginals well in $W_{p, bc}$. This section states the resulting linear program for discrete AOT. Because the structure for the causal, bicausal, and anticausal problems are similar, we only treat the causal case in this section.

Let $\mu, \nu \in \mathcal{P}(\mathbb{R}^N)$ have finite support, and let $n := |\supp(\mu)|, m = |\supp(\nu)|$ and for brevity, denote by $S_\mu := \supp(\mu)$, $S_\nu = \supp(\nu)$. We further write $\mu(x) := \mu(\{x\}), \nu(y) := \nu(\{y\})$. Let $[n] := \{1, \dots, n\}$ for $n\in\mathbb{N}$. For $x \in \mathbb{R}^N$, denote by $x_{1:t}$ the first $t$ entries and by $x_t$ its $t$-th entry, and similarly for $y \in \mathbb{R}^N$.
For all $t$, let $x_{1:t}^{(1)}, x_{1:t}^{(2)}, \dots, x_{1:t}^{(n_{1:t})}$ for some $n_{1:t} \in \mathbb{N}$ be an enumeration of the set $\{x_{1:t} : x \in \supp(\mu)\}$, and $x_{t}^{(1)}, \dots, x_t^{(n_t)}$ for some $n_t \in \mathbb{N}$ be an enumeration of $\{x_t : x \in \supp(\mu)\}$. We similarly define $y_{1:t}^{(i)}$ for $i \in [m_{1:t}]$ and $y_t^{(i)}$ for  $i \in [m_t]$ for some $m_{1:t}, m_t \in \mathbb{N}$.

The following states the LP formulation for causal optimal transport.
\begin{lemma}
	\label{lem:discretization}
	The optimization problem 
	\[
	\min_{\pi \in \Pi_c(\mu, \nu)} \int c \,d\pi
	\]
	is a linear program of the form
	\begin{mini}|l|
		{\substack{\pi_{x, y}:\\ x \in S_\mu, y \in S_\nu}}{\sum_{x \in S_\mu, y \in S_\nu} c(x, y)\,\pi_{x, y} }{}{}
		\addConstraint{\pi_{x, y}}{\in [0, 1],}{x \in S_\mu, y \in S_\nu,}
		\addConstraint{\sum_{x \in S_\mu} \pi_{x, y}}{=\nu(y),}{y \in S_\nu,}
  		\addConstraint{\sum_{y \in S_\nu} \pi_{x, y}}{=\mu(x),}{x \in S_\mu,}
		\addConstraint{\Big(\sum_{\substack{\tilde{x} \in S_\mu:\\ \tilde{x}_{1:t} = x_{1:t}^{(i)}}} \mu(\tilde{x})\Big) \Big(\sum_{\substack{\tilde{x} \in S_\mu:\\\tilde{x}_{1:t} = x_{1:t}^{(i)}\\ \tilde{x}_{t+1} = x_{t+1}^{(k)}}} \sum_{\substack{\tilde{y} \in S_\nu:\\\tilde{y}_{1:t} = y_{1:t}^{(j)}}} \pi_{\tilde{x}, \tilde{y}}\Big)}{}{\nonumber}
		\addConstraint{}{=\Big(\sum_{\substack{\tilde{x} \in S_\mu:\\ \tilde{x}_{1:t} = x_{1:t}^{(i)}\\\\ \tilde{x}_{t+1} = x_{t+1}^{(k)}}} \mu(\tilde{x})\Big) \Big(\sum_{\substack{\tilde{x} \in S_\mu:\\\tilde{x}_{1:t} = x_{1:t}^{(i)}}} \sum_{\substack{\tilde{y} \in S_\nu:\\\tilde{y}_{1:t} = y_{1:t}^{(j)}}} \pi_{\tilde{x}, \tilde{y}}\Big),}{}
		\addConstraint{}{k \in [n_{t+1}], j \in [m_{1:t}], i \in [n_{1:t}], t \in [N-1].}{}
	\end{mini}
	\begin{proof}
		Objective function and marginal constraints are clear. It is left to show that the final set of equality constraints encodes the marginal constraint for $\mu$ and causality. In other words, the last constraint has to encode
		\[
		\pi_{t+1, 0}^{x_{1:t}^{(i)}, y_{1:t}^{(j)}} = \mu_{t+1}^{x_{1:t}^{(i)}}
		\]
		for all $t \in [N-1], i \in n_{1:t}, j \in m_{1:t}$. Since all measures are discrete, this simply means that they put the same weight on each relevant atom $x_{t+1}^{(k)}, k \in [n_{t+1}]$, i.e.
		\[
		\pi_{t+1, 0}^{x_{1:t}^{(i)}, y_{1:t}^{(j)}}(x_{t+1}^{(k)}) = \mu_{t+1}^{x_{1:t}^{(i)}}(x_{t+1}^{(k)}).
		\]
		Noting \begin{align*}
		\mu_{t+1}^{x_{1:t}^{(i)}}(x_{t+1}^{(k)}) = \frac{\mu_{1:t+1}((x_{1:t}^{(i)}, x_{t+1}^{(k)}))}{\mu_{1:t}(x_{1:t}^{(i)})}, 
		\end{align*}
		and 
		\begin{align*}
		\mu_{1:t+1}((x_{1:t}^{(i)}, x_{t+1}^{(k)})) &= \sum_{\substack{\tilde{x} \in S_\mu:\\ \tilde{x}_{1:t} = x_{1:t}^{(i)}\\\\ \tilde{x}_{t+1} = x_{t+1}^{(k)}}} \mu(\tilde{x}), \\
		\mu_{1:t}(x_{1:t}^{(i)}) &= \sum_{\substack{\tilde{x} \in S_\mu:\\ \tilde{x}_{1:t} = x_{1:t}^{(i)}}} \mu(\tilde{x}),
		\end{align*}
		this shows (after multiplying by the denominator) how the terms involving $\mu$ occur in the given equality constraint.
		The terms involving $\pi$ work analogously: Defining $y_{t+1}^{(l)}$, $l \in m_{t+1}$ as an enumeration of $\{y_{t+1} : y \in S_\nu\}$, we have
		\[
		\pi_{t+1, 0}^{x_{1:t}^{(i)}, y_{1:t}^{(j)}}(x_{t+1}^{(k)}) = \sum_{l \in [m_{t+1}]} \pi_{t+1, t+1}^{x_{1:t}^{(i)}, y_{1:t}^{(j)}}(x_{t+1}^{(k)}, y_{t+1}^{(l)})
		\]
		Writing each summand $\pi_{t+1, t+1}^{x_{1:t}^{(i)}, y_{1:t}^{(j)}}(x_{t+1}^{(k)}, y_{t+1}^{(l)})$ as a fraction (analogously to $\mu$ above), we find that the denominator reads
		\[
		\pi_{1:t, 1:t}((x_{1:t}^{(i)}, y_{1:t}^{(j)})) = \Big(\sum_{\substack{\tilde{x} \in S_\mu:\\\tilde{x}_{1:t} = x_{1:t}^{(i)}}} \sum_{\substack{\tilde{y} \in S_\nu:\\\tilde{y}_{1:t} = y_{1:t}^{(j)}}} \pi_{\tilde{x}, \tilde{y}}\Big),
		\]  
		which is independent of $l$ and we can thus multiply by this term (leading to the right hand side term involving $\pi$ in the constraint). Finally, the sum of numerators reads
		\begin{align*}
		\sum_{l \in [m_{t+1}]} \pi_{1:t+1, 1:t+1}((x_{1:t}^{(i)}, x_{t+1}^{(k)}, y_{1:t}^{(k)}, y_{t+1}^{(l)})) &= \sum_{l \in [m_{t+1}]} \sum_{\substack{\tilde{x} \in S_\mu:\\\tilde{x}_{1:t} = x_{1:t}^{(i)}\\ \tilde{x}_{t+1} = x_{t+1}^{(k)}}} \sum_{\substack{\tilde{y} \in S_\nu:\\\tilde{y}_{1:t} = y_{1:t}^{(j)}\\\tilde{y}_{t+1} = y_{t+1}^{(l)}}} \pi_{\tilde{x}, \tilde{y}} \\
		&= \sum_{\substack{\tilde{x} \in S_\mu:\\\tilde{x}_{1:t} = x_{1:t}^{(i)}\\ \tilde{x}_{t+1} = x_{t+1}^{(k)}}} \sum_{\substack{\tilde{y} \in S_\nu:\\\tilde{y}_{1:t} = y_{1:t}^{(j)}}} \pi_{\tilde{x}, \tilde{y}},
		\end{align*}
		which yields the final term involving $\pi$ on the left hand side of the constraint.
	\end{proof}
\end{lemma}

\subsection{Numerical example: approximation of adapted Wasserstein distance}
This section aims to illustrate both the LP formulation for bicausal optimal transport problems given by Lemma \ref{lem:discretization} and the qualitative stability result Theorem \ref{thm:qual_stab}, which justifies the approximation of continuous problems via discretization. The code to reproduce the experiments is available at \url{https://github.com/stephaneckstein/aotnumerics}.

\begin{figure}
	\begin{minipage}[b]{0.5\textwidth}
		\includegraphics[width=1.05\textwidth]{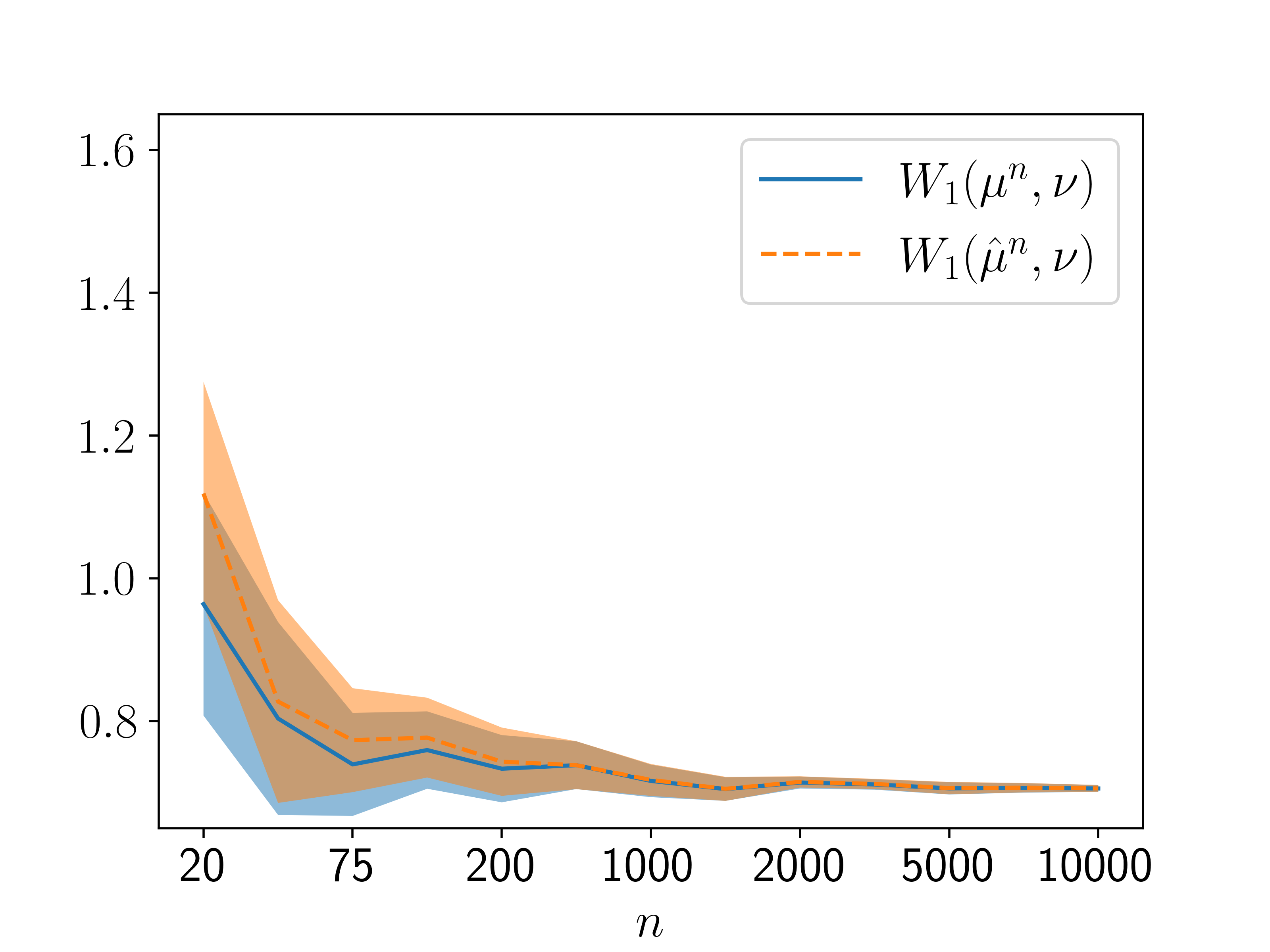}
	\end{minipage}\hspace{-5mm}
	\hfill
	\begin{minipage}[b]{0.5\textwidth} 
		\includegraphics[width=1.05\textwidth]{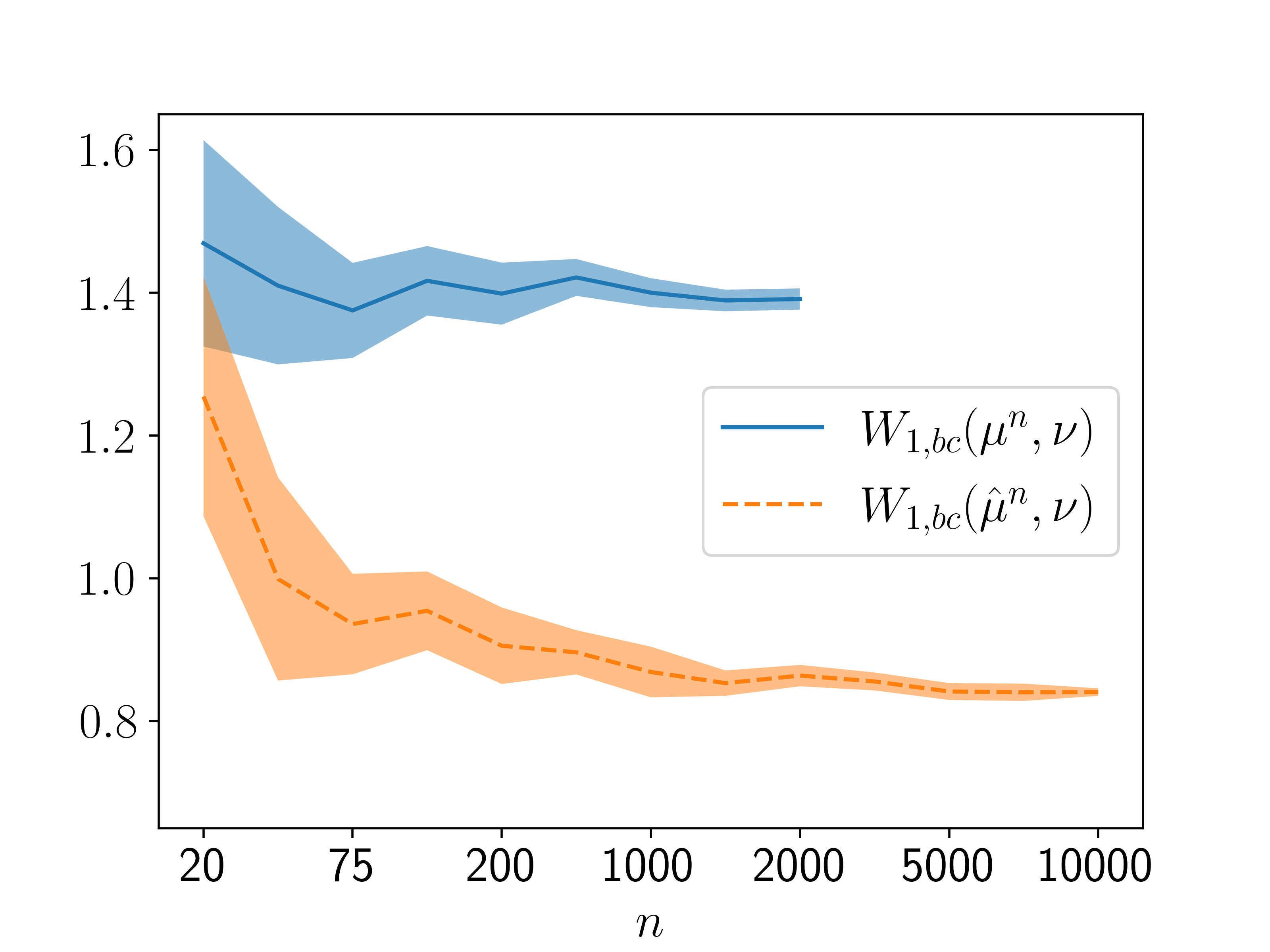}
	\end{minipage}\hspace{-5mm}
	\hfill
	\label{fig:LP_illustration}
 	\caption{Approximation of Wasserstein distance $W_1(\mu, \nu)$ (left) and adapted Wasserstein $W_{1, bc}(\mu, \nu)$ (right) distances between a discrete time Brownian motion $\mu$ and trinomial model $\nu$. The continuous distribution $\mu$ is approximated empirically using an increasing number of sample points $n$, either by the direct empirical measure $\mu^n$ or its adapted version $\hat{\mu}^n$. The reported values are averages over 10 runs using different sample points for $\mu^n$ and $\hat{\mu}^n$, and the error bands are sized according to the standard deviation over these runs. The calculation of $W_{1, bc}(\mu^n, \nu)$ was stopped after $n=2000$ due to long runtimes.}
\end{figure}

Consider two marginals $\mu, \nu \in \mathcal{P}(\mathbb{R}^3)$, where $\mu$ is the discrete time restriction of a Brownian motion starting at $0$, and $\nu$ is the distribution of a trinomial model starting at $0$. That is, the transition kernels $\mu_t^x$ are a normal distribution with mean $x$ and variance $1$, while $\nu_t^x = \frac{1}{3}(\delta_{x-1} + \delta_x + \delta_{x+1})$. The goal is to numerically calculate $W_{1, bc}(\mu, \nu)$ and to compare it with $W_1(\mu, \nu)$.

To calculate these distances, we assume that we know the precise form of $\nu$, but only $n$ i.i.d.~sample paths $x^i = (x_1^{i}, x_2^{i}, x_3^{i}) \sim \mu$, $i=1, \dots, n$. We denote by $\mu^n := \frac{1}{n} \sum_{i=1}^n \delta_{x^i}$ the empirical measure of $\mu$. We emphasize that $W_{1, bc}(\mu^n, \mu)$ does not converge to zero (cf.~\cite{backhoff2020estimating}), and thus to apply Theorem \ref{thm:qual_stab}, $\mu^n$ cannot be used. Indeed, due to the continuous nature of $\mu$, all disintegrations of $\mu^n$ are almost surely deterministic. Thus under $\mu^n$, the transition kernel going forward from $x_2^{(i)}$ is simply the Dirac measure at $x_3^{(i)}$, i.e., $\left(\mu^n\right)_3^{x_2^{(i)}} = \delta_{x_3^{(i)}}$, which is far from the normal distribution given by $\mu_3^{x_2^{(i)}}$. We thus use an adapted empirical measure $\hat{\mu}^n$, where approximately $\sqrt{n}$ points at each time step $t=1, 2, 3$ are clustered together to allow for an empirical approximation which has non-deterministic transition kernels (cf.~\cite[Definition 1.2]{backhoff2020estimating} for a precise construction).

All computations are performed via a direct implementation of the resulting linear program as in Lemma \ref{lem:discretization} (adjusted to the standard and bicausal case) using Gurobi \cite{gurobi}.
The results are illustrated in Figure \ref{fig:LP_illustration}. We see that for the calculation of $W_1(\mu, \nu)$, it does not matter whether one uses $\mu^n$ or $\hat{\mu}^n$, as the approximation using both these empirical measures converges to the same value for large $n$. However, when calculating $W_{1, bc}(\mu, \nu)$, it is crucial to use $\hat{\mu}^n$ so that Theorem \ref{thm:qual_stab} guarantees convergence $W_{1, bc}(\mu^n, \nu) \rightarrow W_{1, bc}(\mu, \nu)$ for $n \rightarrow \infty$. While in this particular case, the conclusion could already be reached with the triangle inequality, Theorem \ref{thm:qual_stab} of course applies to general bicausal optimal transport problems, and not just to $W_{1, bc}(\mu, \nu)$.
Regarding computational times, we note that the most relevant cases, calculation of $W_1(\mu^n, \nu)$ and $W_{1, bc}(\hat{\mu}^n, \nu)$, incurred runtimes within the same order of magnitude, e.g., for $n=1000, 5000, 10000$ the average runtimes to calculate $W_1(\mu^n, \nu)$ were 2.55, 3.90 and 5.25 seconds, while for $W_{1, bc}(\hat{\mu}^n, \nu)$ they were 5.63, 10.22 and 14.38 seconds, all using a single $3.1$GHz CPU. 

\section{Approximation of adapted transport by its entropic version}
\label{sec:conv_reg}
In this section, we first show how the results for the shadow give a very simple methodology for smoothing a given coupling, while the smoothed version keeps the same marginal distributions. Smoothness is understood with respect to the product measure of the marginals in this case. While to the best of the authors' knowledge, the presented method is novel even for non-adapted transport problems, the method extends readily to adapted problems, yielding smoothed versions of couplings with the same marginals \textsl{and} that are adapted (causal, bicausal, or anticausal) in the same sense as the non-smoothed coupling.

The following results are stated under the assumption that the marginals can be discretely approximated in $W_{p, bc}$, which is satisfied in many cases (see, e.g., \cite[Theorem 1.3]{backhoff2020estimating}, \cite[Theorem 5.5]{bartl2021wasserstein}, \cite[Lemma 4.6]{beiglbock2018denseness}). We suspect it holds unconditionally in Polish spaces, which is however not a focus of this work, and thus left as an assumption.
\begin{lemma}
	\label{lemma:shadowshadow}
	Let $p \in [1, \infty)$ and $\mu \in \mathcal{P}_p(\X), \nu \in \mathcal{P}_p(\Y)$. Assume that there exist discrete measures $\mu^n \in \mathcal{P}(\X), \nu^n \in \mathcal{P}(\Y)$ such that \[\lim_{n\rightarrow \infty} W_{p, bc}(\mu^n, \mu) = \lim_{n\rightarrow \infty} W_{p, bc}(\nu^n, \nu) = 0.\]
	Let $\pi \in \Pi_{\bullet}(\mu, \nu)$. Then there exist $\pi^n \in \Pi_\bullet(\mu, \nu)$ such that $\lim_{n\rightarrow \infty} W_{p, c}(\pi, \pi^n) = 0$ and there are constants $C_n > 0$ such that $\frac{d\pi^n}{d P} \leq C_n$ holds $P$-almost surely.
	\begin{proof}
		Fix $n \in \mathbb N$, and consider $W_{p,bc}$-optimal couplings denoted by
		\begin{align*}
		\mu \otimes K^\mu =& \mu^n \otimes \tilde K^\mu \in \Pi_{bc}(\mu,\mu^n),
		\\
		\nu \otimes K^\nu =& \nu^n \otimes \tilde K^\nu \in \Pi_{bc}(\nu,\nu^n).
		\end{align*}
		First, we write $\hat{\pi}^n \in \Pi_\bullet(\mu^n, \nu^n)$ for the $(K^\mu, K^\nu)$-shadow of $\pi$.
		Next, we define $\pi^n \in \Pi_\bullet(\mu,\nu)$ as the $(\tilde{K}^\mu, \tilde{K}^\nu)$-shadow of $\hat{\pi}^n$. 
		Then, by Corollary \ref{cor:triangle} we have
		\[
		W_{p, c}(\pi, \pi^n) \leq W_{p, c}(\pi, \hat{\pi}^n) + W_{p, c}(\hat{\pi}^n, \pi^n),
		\]
		which, together with Lemma \ref{lem:prop_shadow}, yields the first claim. 
		
		To see the second claim, note that, since $\mu^n$ and $\nu^n$ are discrete (and so is the product measure $\mu^n \otimes \nu^n$), there is $C_n > 0$ with
		\[
		\frac{d\hat{\pi}^n}{d \mu^n \otimes \nu^n} \le C_n.
		\]
		Write $\tilde{K} := \tilde{K}^\mu \otimes \tilde{K}^\nu$, then we get
		\[
		\frac{d\hat{\pi}^n \otimes \tilde K}{d(\mu^n \otimes \nu^n) \otimes \tilde K} = \frac{d\hat{\pi}^n}{d \mu^n \otimes \nu^n} \le C_n.
		\]
		Since $\pi^n$ and $P$ are the second marginals of $\hat{\pi}^n \otimes \tilde K$ and $(\mu^n \otimes \nu^n) \otimes \tilde K$, respectively, $\frac{d\pi^n}{d P}$ is a marginal density of $\frac{d\hat{\pi}^n \otimes \tilde K}{d(\mu^n \otimes \nu^n) \otimes \tilde K}$ and thus almost surely bounded by $C_n$ as well.
	\end{proof}
\end{lemma}

While the following theorem is only stated for entropically regularized optimal transport, the strategy of proof extends readily to regularization by other divergences.
\begin{theorem}
	\label{thm:reg_conv}
	Let $p \in [1, \infty)$, $\mu \in \mathcal{P}(\X)$ and $\nu \in \mathcal{P}(\Y)$. 
	Let $c$ be continuous and have growth of order at most $p$. 
	Assume that there exist discrete measures $\mu^n \in \mathcal{P}(\X), \nu^n \in \mathcal{P}(\Y)$ such that \[\lim_{n\rightarrow \infty} W_{p, bc}(\mu^n, \mu) = \lim_{n\rightarrow \infty} W_{p, bc}(\nu^n, \nu) = 0.\]
	Then
	\[
	\lim_{\varepsilon \rightarrow 0} E_{\bullet}^\varepsilon(\mu, \nu, c) = V_\bullet(\mu, \nu, c)
	\]
	and accumulation points of $E_\bullet^\varepsilon$-optimizers (for $\varepsilon \searrow 0$) are optimizers of $V_\bullet(\mu, \nu, c)$.
	\begin{proof}
		Clearly, $E_\bullet^\varepsilon(\mu, \nu, c) \ge V_\bullet(\mu,\nu,c)$, thus
		\[
		\liminf_{\varepsilon \searrow 0} E_\bullet^\varepsilon(\mu, \nu, c) \geq V_\bullet(\mu, \nu, c).
		\]
		To see the reverse direction, pick an optimizer $\pi^*$ of $V_\bullet(\mu, \nu, c)$.
		We choose $\pi^n \in \Pi_{\bullet}(\mu, \nu)$, $n \in \mathbb N$ as the sequence given by Lemma \ref{lemma:shadowshadow} applied w.r.t.\ $\pi^\ast$, then
		\begin{align*}
		\limsup_{\varepsilon \searrow 0} E_\bullet^\varepsilon(\mu,\nu,c)
		&\le
		\liminf_{n \to \infty} 
		\limsup_{\varepsilon \searrow 0} 
		\int c \, d\pi^n + \varepsilon D_{KL}(\pi^n, P)
		\\
		&=
		\liminf_{n \to \infty} \int c \, d\pi^n = \int c \, d\pi^\ast = V_\bullet(\mu,\nu,c),
		\end{align*}
		where we recall that $D_{KL}(\pi^n,P) \le \log(C_n)$ for some constants $C_n > 0$ that are also provided by Lemma \ref{lemma:shadowshadow}.
		Note that the second last equality is due to $W_p$-convergence of $(\pi^n)_{n \in \mathbb N}$ to $\pi^\ast$.
		The final statement of the theorem is a simple consequence of convergence of the values (that was shown above) and the fact that $\Pi_\bullet(\mu,\nu)$ is closed by Lemma \ref{lem:all sets of couplings are compact}.
	\end{proof}
\end{theorem}

\section{Duality}
\label{subsec:duality}
In this section we derive the duality results which will be used to study Sinkhorn's algorithm.
To state the dual formulations of both $V_\bullet$ and $E_\bullet^\varepsilon$ we introduce functional spaces:
First, we define the classical functional spaces
\begin{align*}
\tilde{\mathcal{S}}_{\X} 
:=
\left\{s_{\X}(x, y) = f(x) \colon f \in C_b(\X)\right\}, ~~~
\tilde{\mathcal{S}}_{\Y} 
:= \{s_{\Y}(x, y) = g(y) \colon g \in C_b(\Y)\},
\end{align*}
that determine the marginal constraints in optimal transport.
We denote by $\mathcal S_\X$ and $\mathcal S_\Y$ the sets where $f$ resp.\ $g$ can be even bounded and measurable.
Next, we define functional spaces of continuous functions that characterize causality.
For this reason, we have to refine the topologies on $\X$:
For $\mu \in \mathcal{P}(\X)$, $\nu \in \mathcal{P}(\Y)$ and $t \in \{1,\ldots,N-1\}$, we fix measurable kernels $K_{\mu, t} : \X \rightarrow \mathcal{P}(\X_{t+1:N})$ and $K_{\nu, t} : \Y \rightarrow \mathcal{P}(\Y_{t+1:N})$ given by
\begin{align}
\label{eq:kernel_for_refined_mu}
K_{\mu,t}(x) = \mu_{t+1:N}^{x_{1:t}} ~~\mu\text{-a.s.},
\\
\label{eq:kernel_for_refined_nu}
K_{\nu,t}(y) = \nu_{t+1:N}^{y_{1:t}}~~\nu\text{-a.s.}.
\end{align}
Write $\tau_{\X,\mu}$ and $\tau_{\Y,\nu}$ for the topology on $\X$ and $\Y$, that is obtained by Lemma \ref{lem:refined topology} and makes the kernels \eqref{eq:kernel_for_refined_mu} and \eqref{eq:kernel_for_refined_nu} for all $t=1, \dots, T$ respectively, continuous.
Again, by Lemma \ref{lem:refined topology} we have, for $t \in \{2,\ldots,N\}$,
\begin{multline*}
\tilde{\mathcal A}_{\X,\mu,t} := 
\Big\{ 
f_t(x,y) = a_t(x_{1:t},y_{1:t-1}) -
\int a_t(x_{1:t-1}, \tilde x_t, y_{1:t-1}) \, K_{\mu,t}(x,d\tilde x_{t})
\\
\colon
a_t \in C_b(\X_{1:t} \times \Y_{1:t-1})
\Big\},
\end{multline*}
is a subset of $C_b(\X^\mu \times \Y)$, and
\begin{multline*}
\tilde{\mathcal A}_{\Y,\nu,t} := 
\Big\{ 
g_t(x,y) = a_t(x_{1:t-1},y_{1:t}) -
\int a_t(x_{1:t-1}, y_{1:t-1}, \tilde y_t) \, K_{\nu,t}(y,d\tilde y_{t})
\\
\colon
a_t \in C_b(\X_{1:t-1} \times \Y_{1:t})
\Big\},
\end{multline*}
is a subset of $C_b(\X \times \Y^\nu)$.
We write
\begin{align*}
\tilde{\mathcal A}_{\X,\mu,1} :=& 
\left\{  
f_1(x,y) = a_1(x_1) \colon a_1 \in C_b(\X_1)
\right\},
\\
\tilde{\mathcal A}_{\Y,\nu,1} :=&
\left\{  
g_1(x,y) = a_1(y_1) \colon a_1 \in C_b(\Y_1)
\right\},
\end{align*}
and let $\mathcal{A}_{\X,\mu,t}$ and $\mathcal{A}_{\Y,\nu,t}$ be defined analogously as above when $a_t$ is allowed to be bounded and measurable.
We define the functional spaces
\begin{align*}
\tilde{\mathcal S}_{\X,\mu} :=&
\left\{
s(x,y) = f_1(x_1) + \sum_{t = 2}^N f_t(x_{1:t}, y_{1:t-1}) \colon f_t \in \tilde{\mathcal{A}}_{\X,\mu,t}
\right\},
\\
\tilde{\mathcal S}_{\Y,\nu} :=&
\left\{
s(x,y) = g_1(y_1) + \sum_{t = 2}^N g_t(x_{1:t-1}, y_{1:t}) \colon g_t \in \tilde{\mathcal{A}}_{\Y,\nu,t}
\right\},
\end{align*}
and write again $\mathcal{S}_{\X,\mu}$ and $\mathcal{S}_{\Y,\nu}$ for the corresponding spaces when $\tilde{\mathcal{A}}_{\X,\mu,t}$ and $\tilde{\mathcal{A}}_{\Y,\nu,t}$ are replaced by $\mathcal{A}_{\X,\mu,t}$ and $\mathcal{A}_{\Y,\nu,t}$ respectively.
We finally define 
\begin{align*}
\mathcal H := \mathcal S_\X \oplus \mathcal S_\Y 
&\supseteq
\tilde{\mathcal H} := \tilde{\mathcal S}_\X \oplus \tilde{\mathcal S}_\Y, 
\\
\mathcal H_c := \mathcal S_{\X,\mu} \oplus \mathcal S_\Y
&\supseteq
\tilde{\mathcal{H}}_c := \tilde{\mathcal S}_{\X,\mu} \oplus \tilde{\mathcal S}_\Y,
\\
\mathcal H_{ac} := \mathcal S_\X \oplus \mathcal S_{\Y,\nu}
&\supseteq
\tilde{\mathcal{H}}_{ac} := \tilde{\mathcal S}_\X \oplus \tilde{\mathcal S}_{\Y,\nu},
\\
\mathcal H_{bc} := \mathcal S_{\X,\mu} \oplus \mathcal S_{\Y,\nu}
&\supseteq
\tilde{\mathcal{H}}_{bc} := \tilde{\mathcal S}_{\X,\mu} \oplus \tilde{\mathcal S}_{\Y,\nu}.
\end{align*}

\begin{remark}
	\label{rem:decomp}
	We remark that $\mathcal H \subseteq \mathcal H_\bullet$:
	Indeed, let $f \in L^\infty(\X) \hat= \mathcal S_\X$.
	We set iteratively $a_N(x) := f(x)$, for $t \in \{2,\ldots,N\}$,
	\[
	f_t(x) := a_t(x_{1:t}) - \int a_t(x_{1:t-1}, \tilde x_t) \, \mu^{x_{1:t-1}}_t(d \tilde x_t)
	\in
	\mathcal A_{\X,\mu,t},
	\]
	and $a_{t - 1}(x_{1:t-1}) := \int a_t(x_{1:t-1},\tilde x_t) \, \mu^{x_{1:t-1}}_t(d \tilde x_t)$,
	and $f_1(x) := a_1(x_1)$.
	Then,
	\[
	f(x) = \sum_{t = 1}^N f_t(x) \in \mathcal S_{\X,\mu}.
	\]
	Similarly, we see that $\mathcal S_\Y \subseteq \mathcal S_{\Y,\nu}$.
\end{remark}

The following states the duality results which will used in Section \ref{sec:sinkhorn}.
\begin{proposition}
	\label{prop:duality}
	Let $c \in C_b(\X \times \Y)$, $\mu \in \mathcal P(\X)$ and $\nu \in \mathcal P(\Y)$.
	Then
	\begin{align}
	\label{eq:transport duality}
	V_{\bullet}(\mu, \nu, c) 
	&=
	\sup_{s \in \mathcal{H}_\bullet \colon s\leq c} 
	\int s \,dP,
	\\
	\label{eq:entropic transport duality}
	E^\varepsilon_{\bullet}(\mu, \nu, c) 
	&= 
	\sup_{s \in \mathcal{H}_\bullet} 
	\int s - \varepsilon \exp\Big(\frac{s - c}{\varepsilon}\Big) \,dP + \varepsilon.
	\end{align}
	\begin{proof}
		For classical OT, \eqref{eq:transport duality} is of course well known (see, e.g., \cite{villani2009optimal}).
		We first show \eqref{eq:transport duality} in full generality, where we reason similarly as in \cite[Theorem 4.3]{acciaio2021cournot}:
		By Lemma \ref{lem:causality_char} a probability $\pi \in \mathcal P(\X \times \Y)$ with first marginal $\mu$ is causal if and only if, for all $t \in \{2,\ldots,N\}$,
		\[
		\int f_t \, d\pi = 0 \quad \forall f_t \in \tilde{\mathcal{A}}_{\X,\mu,t}.
		\]
		Hence, $\pi \in \Pi_\bullet(\mu,\nu)$ if and only if
		\[
		\int s \, d\pi = \int s \, dP\quad \forall s \in \tilde{\mathcal H}_\bullet.
		\]
		Note that we may replace $\tilde{\mathcal H}_\bullet$ by $\mathcal H_\bullet$ in the display above, which yields weak duality: the left-hand side in \eqref{eq:transport duality} and \eqref{eq:entropic transport duality} dominates the respective right-hand side.
		
		In order to show the reverse inequalities, we abbreviate
		\begin{align*}
		v(\pi,s) :=& \int c + s \, d\pi - \int s \, dP,
		\end{align*} 
		First, note that $v$ is linear in both arguments. 
		As, for all $s \in \tilde{\mathcal{H}}_\bullet$, $c + s \in C_b(\X^\mu \times \Y^\nu)$ we have (lower semi-)continuity w.r.t.\ weak convergence of measures of $\pi \mapsto v(\pi,s)$ on $\mathcal P(\X^\mu \times \Y^\nu)$.
		Since $\X^\mu$ and $\Y^\nu$ have the same Borel sets as $\X$ and $\Y$ respectively, we find that $\mu$ and $\nu$ are Borel measures on the Polish spaces $\X^\mu$ and $\Y^\nu$.
		Thus, we have that $\Pi(\mu,\nu)$ is a compact subset of $\mathcal P(\X^\mu \times \Y^\nu)$.
		For any $s \in \tilde{\mathcal H}_\bullet$ we have that $c + s \in C_b(\X^\mu \times Y^\nu)$.
		We may apply \cite[Theorem 2.4.1]{AdHe99} and the classical optimal transport duality in order to get
		\begin{align*}
		\inf_{\pi \in \Pi_\bullet(\mu,\nu)}
		\int c \,d\pi
		&=
		\inf_{\pi \in \Pi(\mu,\nu)}
		\sup_{\tilde s \in \tilde{\mathcal H}_\bullet}
		v(\pi,\tilde s)&&=
		\sup_{\tilde s \in \tilde{\mathcal H}_\bullet}
		\inf_{\pi \in \Pi(\mu,\nu)}
		v(\pi,\tilde s)
		\\
		&=
		\sup_{\substack{\tilde s \in \tilde{\mathcal H}_\bullet, s \in \mathcal H, \\ s \le c + \tilde s - \int \tilde s \, dP}}
		\int s \, dP&&=
		\sup_{s \in \mathcal H_\bullet:\,s \leq c} \int s \, dP,
		\end{align*}
		where the first equality follows directly from the characterization of $\Pi_\bullet(\mu,\nu)$ given in Lemma \ref{lem:causality_char} and we used for the final equality that $\mathcal H_\bullet$ is a linear space containing $\mathcal H$ and $\tilde{\mathcal H}_\bullet$.
		
		To show \eqref{eq:entropic transport duality}, we only treat the case $\varepsilon = 1$ for notational convenience, while the general case follows by considering the transformed cost function $\frac{c}{\varepsilon}$.
		Note that
		\[
		D_{\textrm{KL}}(\pi | P) = \sup_{f \in C_b(\X \times \Y)} \int f \, d\pi - \int \exp(f) \, dP + 1,
		\]
		see, e.g., \cite[Proposition 2.12]{Massart.07} (even though arbitrary measurable functions with $\int e^f \, dP = 1$ are used in \cite{Massart.07}, the version with $C_b$ follows by standard approximation arguments  and the observation that the left-hand side dominates the right-hand side as a consequence of the elementary inequality $x y - e^y + 1 \le x \log x$ for all $x > 0$ and $y \in \mathbb R$).
		Moreover, the term
		\begin{align*}
		v_{ent}(\pi,f) 
		:=&
		\int c + f \, d\pi - \int \exp(f) \, dP + 1
		\\
		=& \int g \, d\pi - \int \exp(g - c) \, dP + 1,
		\end{align*}
		(where we substituted $g := f - c$ in the second line)
		is, for fixed $f \in C_b(\X \times \Y)$, convex and lower semicontinuous in the variable $\pi \in \mathcal P(\X^\mu \otimes \Y^\nu)$ and,
		for fixed $\pi \in \mathcal P(\X^\mu \otimes Y^\nu)$, concave in the variable $f \in C_b(\X \times \Y)$.
		Thus by again applying \cite[Theorem 2.4.1]{AdHe99}, we find
		\begin{align*}
		E^\varepsilon_\bullet(\mu,\nu) 
		&=
		\inf_{\pi \in \Pi_\bullet(\mu,\nu)}
		\sup_{f \in C_b(\X \times \Y)}
		v_{ent}(\pi,f)
		\\
		&=
		\sup_{g \in C_b(\X \times \Y)}
		\inf_{\pi \in \Pi_\bullet(\mu,\nu)}
		\int g \, d\pi - \int \exp(g- c) \, dP + 1
		\\
		&=
		\sup_{g\in C_b(\X \times \Y)}
		\sup_{\substack{s \in \mathcal H_\bullet}{s \le g}}
		\int s - \exp(g - c) \, dP + 1
		\\
		&=
		\sup_{s \in \mathcal H_\bullet}
		\int s - \exp(s - c) \, d P + 1. \qedhere
		\end{align*}

	\end{proof}
\end{proposition}
Note that the integral $\int s \,dP$ simplifies in all cases, for instance if $s \in \mathcal{H}_{bc}$, then $\int s \,dP = \int f_1 \,d\mu_1 + \int g_1 \,d\nu_1$, where $s = s_{\X} + s_\Y$ and the functions $f_1$ and $g_1$ correspond to $s_\X$ and $s_\Y$, respectively as in the definition of $\mathcal{S}_{\X, \mu}$ and $\mathcal{S}_{\Y, \nu}$.

\section{Sinkhorn's algorithm}
\label{sec:sinkhorn}
Within this section, we continue working with $\varepsilon = 1$, while the general case may be recovered by scaling the cost function $c$.
We start the section by defining Sinkhorn's algorithm and showing how the respective iterations can be computed both in the primal and the dual formulation. Observe that
\[
\argmin_{\pi \in \Pi_{\bullet}(\mu, \nu)} \int c\,d\pi + D_{\rm KL}(\pi, P) = \argmin_{\pi \in \Pi_{\bullet}(\mu, \nu)} D_{\rm KL}(\pi, P_c),
\]
where $P_c \in \mathcal{P}(\X \times \Y)$ is defined via 
\begin{equation}
\label{eq:defpc}
\frac{d P_c}{dP}(x, y) := \frac{\exp(-c(x, y))}{\int \exp(-c) \,dP},
\end{equation}
since the value of $D_{\rm KL}(\pi,P_c)$ differs from $\int c \, d\pi + D_{\rm KL}(\pi,P)$ only by the constant $\log(\int \exp(-c) \,dP)$.
Recall that the unique optimizer exists by Remark \ref{rem:existence_optimizers}.
In other words, computing the optimizer of the problem $E_\bullet(\mu, \nu, c)$ is the same as computing the projection of $P_c$ onto the set $\Pi_{\bullet}(\mu, \nu)$ with respect to relative entropy.

The basic premise of Sinkhorn's algorithm is to regard the set of measures we optimize over, $\Pi_{\bullet}(\mu, \nu)$, as an intersection of simpler sets. Instead of computing the projection of $P_c$ onto $\Pi_{\bullet}(\mu, \nu)$ directly, the algorithm consists of iteratively projecting onto the simpler sets.
For $\mu \in \mathcal{P}(\X), \nu \in \mathcal{P}(\Y)$, define
\begin{align*}
\Pi_{c}(\mu, *) &= \cup_{\tilde{\nu} \in \mathcal{P}(\Y)} \Pi_c(\mu, \tilde{\nu}),\\ 
\Pi(*, \nu) &= \cup_{\tilde{\mu} \in \mathcal{P}(\X)} \Pi(\tilde{\mu}, \nu),\\
\Pi_{ac}(*, \nu) &= \cup_{\tilde{\mu} \in \mathcal{P}(\X)} \Pi_{ac}(\tilde{\mu}, \nu).
\end{align*}
These sets will serve as the simple building blocks of the sets we project onto.
We focus on the causal and bicausal problem in the following. We are thus interested in computing
\begin{align*}
\argmin_{\pi \in \Pi_{c}(\mu, \nu)} D_{\rm KL}(\pi, P_c) ~ \text{ and } ~ \argmin_{\pi \in \Pi_{bc}(\mu, \nu)} D_{\rm KL}(\pi, P_c).
\end{align*}

\begin{definition}[Sinkhorn steps]
	\label{def:sinkhornsteps}
	Let $\mathcal{Q}_1 = \Pi_c(\mu, *)$ and $\mathcal{Q}_2 \in \{\Pi(*, \nu), \Pi_{ac}(*, \nu)\}$.
	Let $\pi^{(0)} := P_c$ be as defined in \eqref{eq:defpc} and for $k \geq 0$, 
	\begin{align*}
	\pi^{(k+\frac{1}{2})} &:= \argmin_{\pi \in \mathcal{Q}_1}D_{\rm KL}(\pi, \pi^{(k)}), \\
	\pi^{(k+1)} &:= \argmin_{\pi \in \mathcal{Q}_2} D_{\rm KL}(\pi, \pi^{(k+\frac{1}{2})}).
	\end{align*}
	We refer to a step in the algorithm as \emph{causal step} if we project on $\mathcal Q_1$, \emph{normal step} if we project on $\mathcal Q_2 = \Pi(\ast,\nu)$ and \emph{anticausal step} if we project on $\mathcal Q_2 = \Pi_{ac}(\ast,\nu)$.
\end{definition}
The choice of $\mathcal{Q}_2$ determines which problem, either causal or bicausal OT, is solved. If $\mathcal{Q}_2 = \Pi(*, \nu)$, then the causal problem is solved, and if $\mathcal{Q}_2 = \Pi_{ac}(*, \nu)$, then the bicausal problem is solved. This choice remains fixed throughout the whole algorithm. We note that the existence of the minimizers is shown in Lemma \ref{lem:proj_primal} for causal and anticausal steps, respectively well known and tractable (see, e.g., \cite{nutz2021introduction}) for the normal step.
While tractability of projecting onto the set $\Pi(*, \nu)$ is well known, the following results establish the analogue for $\Pi_{c}(\mu, *)$.
The symmetric result for $\Pi_{ac}(*, \nu)$ can directly be inferred, but is not stated explicitly.

\begin{lemma}[Causal projection: primal version]
	\label{lem:proj_primal}
	Let $\gamma \in \mathcal{P}(\X \times \Y)$ be such that \[\inf_{\tilde\pi \in \Pi_c(\mu, *)} D_{\rm KL}(\tilde\pi, \gamma) < \infty.\]
	Recursively, we define functions $f_t : \X_{1:t} \times \Y_{1:t} \rightarrow \mathbb{R}$ and a measure $\pi \in \mathcal{P}(\X \times \Y)$ through its disintegrations:
	Set $f_N \equiv 0$ and, for $t=N-1, \dots, 1$,
	\begin{align*}
	\pi_{t+1, 0}^{X_{1:t}, Y_{1:t}} &:= \mu_{t+1}^{X_{1:t}}, \quad
	\frac{d\pi_{0, t+1}^{X_{1:t+1}, Y_{1:t}}}{d\gamma_{0, t+1}^{X_{1:t+1}, Y_{1:t}}} := \frac{\exp(-f_{t+1})}{\int \exp(-f_{t+1}) \,d\gamma_{0, t+1}^{X_{1:t+1}, Y_{1:t}}}, \\
	f_{t} &:= D_{\rm KL}(\pi_{t+1, t+1}^{X_{1:t}, Y_{1:t}}, \gamma_{t+1, t+1}^{X_{1:t}, Y_{1:t}}) + \int f_{t+1} \,d\pi_{t+1, t+1}^{X_{1:t}, Y_{1:t}},
	\end{align*}
	and 
	\[
	\pi_{1, 0} 
	:= 
	\mu_{1},
	\quad
	\frac{d\pi_{0, 1}^{X_1, Y_0}}{d\gamma_{0, 1}^{X_1, Y_0}} := \frac{\exp(-f_{1})}{\int \exp(-f_{1}) \,d\gamma_{0, 1}^{X_{1}, Y_{0}}}.
	\]
	Then, we have
	\[
	\pi = \argmin_{\tilde\pi \in \Pi_c(\mu, \ast)} D_{\rm KL}(\tilde{\pi}, \gamma).
	\]
	
	\begin{proof}
		First note that by plugging in the recursion for $f_t$, we can rewrite for $t < N$
		\begin{equation}
		\label{eq:proj_primal_ft}
		f_t(x_{1:t}, y_{1:t}) 
		= D_{\rm KL}\big(\pi_{t+1:T, t+1:T}^{x_{1:t}, y_{1:t}}, \gamma_{t+1:T, t+1:T}^{x_{1:t}, y_{1:t}}\big).
		\end{equation}
		Let $\tilde \pi \in \Pi_c(\mu,\ast)$ and we show that $D_{KL}(\tilde \pi , \gamma) \ge D_{KL}(\pi , \gamma)$.
		Using the chain rule of the relative entropy, see \cite[Theorem C.3.1]{dupuis2011weak}, we find 
		\begin{align}
		\label{eq:proj_primal_chain_rule1}
		D_{\rm KL}(\tilde \pi, \gamma) 
		= 
		D_{\rm KL}(\tilde \pi_{1, 1}, \gamma_{1, 1}) 
		+ \sum_{t=1}^{N-1} 
		\int D_{\rm KL}(\tilde \pi_{t+1, t+1}^{x_{1:t}, y_{1:t}}, \gamma_{t+1, t+1}^{x_{1:t}, y_{1:t}})
		\,\tilde \pi_{1:t, 1:t}(dx_{1:t}, dy_{1:t}),
		\end{align}
		where, again by the chain rule and as $\tilde \pi$ is causal,
		\begin{multline}
		\label{eq:proj_primal_chain_rule2}
		D_{KL}(\tilde \pi^{x_{1:t}, y_{1:t}}_{t + 1, t + 1}, \gamma^{x_{1:t}, y_{1:t}}_{t + 1, t + 1})
		\\
		=
		D_{KL}(\mu^{x_{1:t}}_{t+1}, \gamma_{t+1,0}^{x_{1:t},y_{1:t}})
		+
		\int
		D_{KL}(\tilde \pi^{x_{1:t+1},y_{1:t}}_{0,t+1},\gamma_{0,t+1}^{x_{1:t+1},y_{1:t}})
		\,\mu^{x_{1:t}}_{t+1}(dx_{t+1}).
		\end{multline}
		Let $t = 0,\ldots, N-1$ and set $\pi^N := \tilde \pi$, $\pi^t := \tilde \pi_{1:t,1:t} \otimes \pi_{t+1:N, t+1:N}^{X_{1:t}, Y_{1:t}}$, and $\pi^0 := \pi$.
		We show
		\begin{equation}
		\label{eq:proj_primal_decreasing_entropy}
		D_{KL}(\pi^{t + 1}, \gamma) - D_{KL}(\pi^t, \gamma) \ge 0.
		\end{equation} 
		Since $\pi^N = \tilde\pi$ and $\pi^0 = \pi$, this will yield the claim.
		Combining \eqref{eq:proj_primal_ft} and \eqref{eq:proj_primal_chain_rule1} imply
		\begin{multline}
		\label{eq:calcpit}
		D_{\rm KL}(\pi^t , \gamma) =
		D_{\rm KL}(\tilde \pi_{1,1}, \gamma_{1,1})
		+
		\sum_{s = 1}^{t-1} 
		\int D_{\rm KL}(\tilde \pi_{s+1, s+1}^{x_{1:s}, y_{1:s}}, \gamma_{s+1, s+1}^{x_{1:s}, y_{1:s}})
		\,\tilde \pi_{1:s, 1:s}(dx_{1:s}, dy_{1:s})
		\\
		+
		\int f_t(x_{1:t}, y_{1:t}) \, \tilde \pi_{1:t,1:t}(dx_{1:t}, dy_{1:t}).
		\end{multline}
		To shorten notation, we define the kernels
		\[
		K:= \pi^{X_{1:t-1}, Y_{1:t-1}}_{t,t}, \quad
		\tilde K := \tilde \pi^{X_{1:t-1}, Y_{1:t-1}}_{t,t}, \quad
		R := \gamma^{X_{1:t-1}, Y_{1:t-1}}_{t,t}.
		\]
		Note that the definition of $K = \pi^{X_{1:t-1}, Y_{1:t-1}}_{t,t}$ in the statement of the present lemma implies
		\[
		K
		=
		\argmin \left \{ D_{\rm KL}(\eta, R) + \int f_t \, d\eta \colon \eta \in \Pi_c(\mu_t^{X_{1:t-1}}, *) \right\},
		\]
		which is understood point-wise (for each $X_{1:t-1} = x_{1:t-1}, Y_{1:t-1} = y_{1:t-1}$), and
		can easily deduced by changing the reference measure and using the chain rule as in \eqref{eq:proj_primal_chain_rule2}.
		Hence, using \eqref{eq:calcpit} and the definition of $f_t$, we get
		\begin{align*}
		D_{\rm KL}(\pi^t, \gamma) - &D_{\rm KL}(\pi^{t-1}, \gamma)
		\\
		&=
		\int
		\Big(
		D_{\rm KL}(\tilde K, R) - D_{\rm KL}(K, R) + \int f_t \, d(\tilde K - K) 
		\Big)
		\, d\tilde \pi_{1:t-1,1:t-1}
		\\
		&\ge 0,
		\end{align*}
		where the inequality uses that $\tilde K \in \Pi(\mu_t^{X_{1:t-1}},\ast)$ and optimality of $K$. This shows \eqref{eq:proj_primal_decreasing_entropy} and thus the claim. \qedhere

	\end{proof}
\end{lemma}

\begin{corollary}
	\label{cor:invarianceopti}
	Let $\gamma, \hat \gamma \in \mathcal P(\mathcal X \times \mathcal Y)$ be equivalent, i.e.\ $\gamma \ll \hat \gamma$ and $\hat \gamma \ll \gamma$, such that
	\begin{align*}
	\inf_{\tilde \pi \in \Pi_c(\mu,\ast)} D_{\textrm{KL}}(\tilde \pi, \gamma) < \infty, ~~~ 
	\inf_{\tilde \pi \in \Pi_c(\mu,\ast)} D_{\textrm{KL}}(\tilde \pi, \hat \gamma) < \infty,
	\end{align*}
	and for $t = 1,\ldots,N-1$, assume $\gamma$-a.s.\
	\begin{gather}
	\label{eq:causal_dens1}
	\frac{d\hat \gamma^{X_{1:t},Y_{1:t}}_{t+1,t+1}}{d\gamma^{X_{1:t},Y_{1:t}}_{t+1,t+1}}
	=
	\exp( h_{t+1}(X_{1:t+1},Y_{1:t}) ),
	\\
	\label{eq:causal_dens2}
	\int h_{t+1}(X_{1:t},\tilde x_{t + 1}, Y_{1:t}) \, \mu_{t+1}^{X_{1:t}}(d\tilde x_{t+1})  = 0,
	\end{gather}
	for measurable functions $h_{t+1} \colon \mathcal X_{1:t+1} \times \mathcal Y_{1:t} \to \mathbb R$
	and 
	\[
	\frac{d \hat \gamma_{1,1}}{ d \gamma_{1,1}} = \exp(h_1(X_1)),
	\]
	for $h_1 \colon \mathcal X_1 \to \mathbb R$ measurable.
	Then, $\pi$ given as in Lemma \ref{lem:proj_primal} satisfies
	\[
	\pi = \argmin_{\tilde \pi \in \Pi_c(\mu,\ast)} D_{\textrm{KL}} (\tilde \pi, \gamma) 
	= \argmin_{\tilde \pi \in \Pi_c(\mu,\ast)} D_{\textrm{KL}} (\tilde \pi, \hat \gamma).
	\]
\end{corollary}

\begin{proof}
	By \eqref{eq:causal_dens1} we find that
	\[
	\frac{d\hat\gamma^{X_{1:t+1}, Y_{1:t}}_{0,t+1}}{d\gamma^{X_{1:t+1},Y_{1:t}}_{0,t+1}} = 1,
	\]
	which yields $\gamma$-almost surely (and $\hat \gamma$-almost surely) for $t=0, \dots, T-1$
	\begin{equation}
	\label{eq:cor1}
	\frac{d\pi^{X_{1:t+1},Y_{1:t}}_{0,t+1}}{d \hat \gamma^{X_{1:t+1},Y_{1:t}}_{0,t+1}}
	=
	\frac{d\pi^{X_{1:t+1},Y_{1:t}}_{0,t+1}}{d\gamma^{X_{1:t+1},Y_{1:t}}_{0,t+1}}.
	\end{equation}
	We have by \eqref{eq:causal_dens2} and as $\pi \in \Pi_c(\mu, \ast)$ that
	\begin{align}
	\nonumber
	D_{\textrm{KL}} (\pi^{X_{1:t},Y_{1:t}}_{t+1,t+1},\hat \gamma^{X_{1:t},Y_{1:t}}_{t+1,t+1}) 
	&= 
	D_{\textrm{KL}}(\pi^{X_{1:t},Y_{1:t}}_{t+1,t+1},\gamma^{X_{1:t},Y_{1:t}}_{t+1,t+1})
	 - 
	\int h_{t + 1}(X_{1:t}, \tilde x_{t+1}, Y_{1:t}) \, d\pi^{X_{1:t},Y_{1:t}}_{t+1,t+1}
	\\
	\nonumber
	&=
	D_{\textrm{KL}}(\pi^{X_{1:t},Y_{1:t}}_{t+1,t+1},\gamma^{X_{1:t},Y_{1:t}}_{t+1,t+1})
	-
	\int h_{t + 1}(X_{1:t}, \tilde x_{t+1}, Y_{1:t}) \, d\mu^{X_{1:t}}_{t+1}(\tilde x_{t+1})
	\\
	\label{eq:cor2}
	&=
	D_{\textrm{KL}}(\pi^{X_{1:t},Y_{1:t}}_{t+1,t+1},\gamma^{X_{1:t},Y_{1:t}}_{t+1,t+1}).
	\end{align}
	Comparing \eqref{eq:cor1} and \eqref{eq:cor2} with the recursion in Lemma \ref{lem:proj_primal} yields that $\pi$ also satisfies the same recursion when $\gamma$ is replaced by $\hat \gamma$, which shows the claim.
\end{proof}

Recall that $\mathcal{S}_{\X, \mu}$ is defined as the dual space to the constraint $\Pi_{c}(\mu, *)$, cf.~Section \ref{subsec:duality}.
\begin{lemma}[Causal projection: dual version]
	\label{lem:proj_dual}
	Let $\mu \in \mathcal P(\X)$, $\nu \in \mathcal P(\Y)$, and
	$\gamma \in \mathcal{P}(\X \times \Y)$ have density $\frac{d\gamma}{dP} = \exp(g)$ for some measurable and bounded function $g$.
	Let $\pi \in \Pi_c(\mu, *)$ and $s_\X \in \mathcal{S}_{\X, \mu}$ be such that
	\begin{equation}
	\label{eq:optimal causal density}
	\frac{d\pi}{dP} = \exp(s_\X + g).
	\end{equation}
	Then, $\pi$ is the unique minimizer of $\min_{\tilde{\pi} \in \Pi_{c}(\mu, *)} D_{\rm KL}(\tilde{\pi}, \gamma)$. 
	
	Moreover, $s_\X$ can be decomposed as a sum, $s_\X = \sum_{t = 1}^N f_t$,
	where $f_t \in \mathcal A_{\X,\mu,t}$ are explicitly given by the following recursion: 
	Set $g_N = g$ and, for $t=N-1, \dots, 1$,
	\begin{align*}
	a_{t+1} &:= - \log \int \exp(g_{t+1}) \,d\nu_{t+1}^{Y_{1:t}} \\
	g_t &:= - \int a_{t+1} \,d\mu_{t+1}^{X_{1:t}} \\
	f_{t+1} &:= a_{t+1} + g_t
	\end{align*}
	and $f_1 = -\log \int \exp(g_1) \,d\nu_1$.
	\begin{proof}
		First we show optimality of $\pi$ when its density takes the form as in \eqref{eq:optimal causal density}. 
		Obviously, $\frac{d\pi}{d\gamma} = \exp(s_{\X})$.
		Thus, for $\tilde\pi \in \Pi_c(\mu, \ast)$ with $D_{\rm KL}(\tilde \pi, \gamma) < \infty$ we get
		\begin{align}
		\label{eq:optimal causal density proof1}
		\begin{split}
		D_{\rm KL}(\tilde\pi, \gamma) - D_{\rm KL}(\tilde\pi, \pi) 
		&= 
		D_{\rm KL}(\tilde\pi, \gamma) - \int \log \left(\frac{d\tilde\pi}{d\gamma} \right) \,d\tilde\pi 
		+ \int \log \left(\frac{d\pi}{d\gamma} \right) \,d\tilde\pi \\
		&= \int s_\X \,d\tilde\pi = \int s_\X \,d\pi = D_{\rm KL}(\pi, \gamma)
		\end{split}
		\end{align}
		which implies that $\pi = \argmin_{\tilde\pi \in \Pi_c(\mu, *)} D_{\rm KL}(\tilde\pi, \gamma)$.
		Further, if $\pi' \in \Pi_c(\mu,\ast)$ was another minimizer, then \eqref{eq:optimal causal density proof1} yields
		that $D_{\rm KL}(\pi', \pi) = 0$, thus, $\pi' = \pi$.
		
		
		Concerning the second assertion, let $\hat \pi \in \Pi_c(\mu,\ast)$ and rewrite the left-hand side in \eqref{eq:optimal causal density} as
		\begin{align}
		\label{eq:optimal causal density proof2}
		\frac{d\hat \pi}{dP} 
		&= 
		\Big( 
		\frac{d\hat \pi_{1, 0}}{d\mu_1} 
		\frac{d\hat \pi_{0, 1}^{X_1, Y_0}}{d\nu_1}
		\Big) 
		\prod_{t=1}^{N-1} 
		\Big( 
		\frac{d\hat \pi_{t+1, 0}^{X_{1:t}, Y_{1:t}}}{d\mu_{t+1}^{X_{1:t}}} 
		\frac{d\hat \pi_{0, t+1}^{X_{1:t+1}, Y_{1:t}}}{d\nu_{t+1}^{Y_{1:t}}} 
		\Big)
		\\
		\label{eq:optimal causal density proof3}
		&= 
		\Big(
		\frac{d\hat \pi_{0, 1}^{X_1, Y_0}}{d\nu_1}
		\Big) 
		\prod_{t=1}^{N-1} 
		\Big(
		\frac{d\hat \pi_{0, t+1}^{X_{1:t+1}, Y_{1:t}}}{d\nu_{t+1}^{Y_{1:t}}} 
		\Big),
		\end{align}
		where the second equality holds if and only if $\hat \pi \in \Pi_c(\mu,\ast)$, since then, for $t = 1,\ldots,N-1$,
		\[
		\frac{d\hat \pi_{1, 0}}{d\mu_1} = 1 = \frac{d\hat \pi_{t+1, 0}^{X_{1:t}, Y_{1:t}}}{d\mu_{t+1}^{X_{1:t}}}.
		\]
		Let $f_t$ and $g_t$ be given as in the statement of the lemma.
		Moreover, set $\hat s_\X := \sum_{t = 1}^N f_t$ which is by construction in $\mathcal S_{\X,\mu}$, then we have
		\begin{align}
		\label{eq:optimal causal density proof4}
		\exp(\hat s_\X + g) 
		= 
		\Big(
		\exp(g_1 + f_1)
		\Big) 
		\prod_{t=1}^{N-1} 
		\Big(
		\exp(g_{t+1} + a_{t+1})
		\Big).
		\end{align}
		Note that $\exp(g_1 + f_1) \in L^\infty(\X_1)$ is a density w.r.t.\ $\nu_1$, whereas for $t = 2,\ldots, N$, $\exp(g_{t} + a_{t}) \in L^\infty(\X_{1:t+1} \times \Y_{1:t})$ are densities w.r.t.\ $\nu_{t+1}^{Y_{1:t}}$.
		Thus, by defining $\hat \pi$ as in \eqref{eq:optimal causal density} with $\hat s_\X$,
		we obtain from comparing \eqref{eq:optimal causal density proof2}, \eqref{eq:optimal causal density proof3}, and \eqref{eq:optimal causal density proof4} that
		\[
		\frac{d\hat \pi_{0,1}^{X_1,Y_1}}{d\nu_1} = \exp(g_1 + f_1),
		\quad
		\frac{d\hat \pi_{0,t+1}^{X_{1:t+1},Y_{1:t}}}{d\nu_{t+1}^{Y_{1:t}}} = \exp(g_t + f_t),
		\]
		and $\hat \pi \in \Pi_c(\mu,\ast)$.
		Finally, by uniqueness that was shown in the first part, we derive that any $s_\X \in \mathcal S_{\X,\mu}$ satisfying \eqref{eq:optimal causal density} has to coincide with $\hat s_\X$, hence, $s_\X = \sum_{t = 1}^N f_t$.
	\end{proof}
\end{lemma}

\subsection{The causal case via stability}
In this section, we give a quick argument, following the lines of \cite[Section 3.4]{eckstein2021quantitative}, to prove the convergence of Sinkhorn's algorithm via the stability results obtained in Section \ref{subsec:stab}. 
This is based on the argument that iterations of the algorithm are each optimal - albeit for their own marginals instead of the true ones.
As it is, this line of argument only works in the causal case but not the bicausal one.
The reason is that for causal OT, the causality constraint is satisfied every second step during Sinkhorn's algorithm. For bicausal OT on the other hand, each iteration may fail to satisfy a bicausality constraint, and hence the respective stability results cannot be applied directly.
In the following proposition, notice that for a bounded metric (and thus bounded cost function), no assumptions on the marginals are required at all.
\begin{proposition}
	\label{prop:conv_causal_sink}
	Let $p \in [1, \infty)$ and assume $c$ is continuous and has growth of order at most $p$. Denote by $\pi^*$ the optimizer of $E_c(\mu, \nu, c)$ and let $\pi^{(k)}, \pi^{(k+\frac{1}{2})}$ be the causal iterations of Sinkhorn's algorithm as defined in Definition \ref{def:sinkhornsteps}. If $\nu$ has finite exponential moments of order $p$,\footnote{Which is defined as $\int \exp(\alpha d_\X(\hat{x}, x)^p \,\mu(dx) < \infty$ for some $\alpha > 0, \hat{x} \in \X$.} then $\pi^{(k+\frac{1}{2})}$ converges in $W_p$ to $\pi^*$ as $k\rightarrow \infty$, $k \in \mathbb{N}$. If further $\mu$ has finite exponential moments of order $p$ as well, then $\pi^{(k)}$ also converges in $W_p$ to $\pi^*$ as $k \rightarrow \infty$.
	\begin{proof}
		Write $\pi^{(l)}$ for $l \in \{0, \frac{1}{2}, 1, 1 + \frac{1}{2}, 2, 2 + \frac{1}{2}, \dots\}$.
		To ease notation, we write $\mu^{(l)} = \pi^{(l)}_{1:N, 0}$ and $\nu^{(l)} = \pi^{(l)}_{0, 1:N}$.
		By the Pythagorean theorem for the relative entropy, see \cite[Equation (3.1) and Corollary 3.1]{csiszar1975divergence}, we find for any $\pi \in \Pi_\bullet(\mu,\nu)$ which has finite entropy w.r.t.\ $P$ (and therefore also w.r.t.\ $P_c = \pi^{(0)}$) that
		\begin{equation}
		\label{eq:pythagorean_thm}
		D_{\textrm{KL}}(\pi, P_c) = D_{\textrm{KL}}(\pi, \pi^{(n)}) + \sum_{l\in\{\frac{1}{2}, 1, 1+\frac{1}{2}, \dots, n\}} D_{\textrm{KL}}( \pi^{(l)}, \pi^{(l - \frac{1}{2})} ),
		\end{equation}
		where $n \in \{ \frac12, 1, 1 + \frac12, \ldots \}$.
		A quick calculation (cf.~\cite[Proposition 2.1]{ruschendorf1995convergence}) reveals 
		\[D_{\rm KL}(\mu^{(l)}, \mu) \rightarrow 0~\text{ and } ~D_{\rm KL}(\nu^{(l)}, \nu) \rightarrow 0.\] By the tail assumptions on $\mu$ and $\nu$ and by \cite[Corollary 2.3]{bolley2005weighted}, this implies $W_p$ convergence of the marginals (and thus in particular convergence of the $p$-th moments). Further, by Pinsker's inequality and Lemma \ref{lem:AV_TV}, this convergence also holds in $AV$. Convergence in $AV$ (which directly implies $W_{p, bc}$ convergence for a bounded metric) together with convergence of the $p$-th moments implies convergence in $W_{p, bc}$, and thus
		\[W_{p, bc}(\mu^{(l)}, \mu) \rightarrow 0~\text{ and }~W_{p, bc}(\nu^{(l)}, \nu) \rightarrow 0.\]
		Notably, for $l = k + \frac{1}{2}$, this convergence holds without the tail assumption on $\mu$ (since for these iterations the first marginal  of $\pi^l$ equals $\mu$).
		
		We turn towards applying our stability result for $l = k+\frac{1}{2}$ for some $k \in \mathbb{N}$. 
		We set
		\[
		s_{\X, k} 
		= \sum_{j=0}^k \log \left( \frac{d\pi^{(j+\frac{1}{2})}}{d\pi^{(j)}}\right),
		\quad
		s_{\Y, k} 
		= \sum_{j=1}^k \log \left( \frac{d\pi^{(j)}}{d\pi^{(j-\frac{1}{2})}} \right)
		= \sum_{j=1}^k \log \left( \frac{d\nu}{d\nu^{(j-\frac{1}{2})}} \right),
		\]
		and obtain $\frac{d\pi^{(k)}}{d P_c} = \exp(s_{\X, k} + s_{\Y, k})$, which we abbreviate as $\pi^{(k)} = \exp(s_{\X, k} + s_{\Y, k}) P_c$.
		We remark that $s_{\Y,k} \in L^1(\nu)$ as a consequence of \eqref{eq:pythagorean_thm}.
		Further, the respective densities $\frac{d\pi^{(j+\frac{1}{2})}}{d\pi^{(j)}}$ are given by Lemma \ref{lem:proj_primal}. 
		We obtain for some constants $k_1, k_2 \in \mathbb{R}$
		\begin{align}
		\label{eq:optichain1}
		\begin{split}
		\pi^{(l)} &= \argmin_{\pi \in \Pi_c(\mu, *)} D_{\rm KL}(\pi, \pi^{(k)}) \\
		&= \argmin_{\pi \in \Pi_c(\mu, *)} D_{\rm KL}(\pi, \exp(s_{\X, k} + s_{\Y, k}) P_c)\\
		&\stackrel{(1)}{=} \argmin_{\pi \in \Pi_c(\mu, *)} D_{\rm KL}(\pi, \exp(k_1 + s_{\Y, k}) P_c)\\
		&= \argmin_{\pi \in \Pi_c(\mu, *)} \int c \,d\pi +  D_{\rm KL}(\pi, \exp(k_2 + s_{\Y, k}) P),
		\end{split}
		\end{align}
		where $(1)$ follows by Corollary \ref{cor:invarianceopti}, where we note that the problem on the left hand side of the equality is finite by plugging in the optimizer $\pi^{(l)}$, and the right hand side is finite by plugging in $P = \mu \otimes \nu$, while the form of $s_{\X, k}$ coming from Lemma \ref{lem:proj_primal} satisfies the assumptions of Corollary \ref{cor:invarianceopti}. To continue, we denote by $\tilde \nu$ the measure on $\Y$ that satisfies $\mu \otimes \tilde\nu = \exp(k_2 + s_{\Y, k}) P$, and find
		\begin{align*}
		D_{\rm KL}(\nu^{(l)}, \tilde\nu) &\leq D_{\rm KL}(\pi^{(l)}, \mu \otimes \tilde\nu) \\ &= D_{\rm KL}(\pi^{(l)}, \exp(k_1 + s_{\Y, k}) P_c) + (k_1 - k_2) + \int \log\left(\frac{dP_c}{dP}\right) \,d\pi^{(l)} < \infty
		\end{align*}
		by the integrability condition on $c$.
		Following up on \eqref{eq:optichain1}, we obtain
		\begin{align*}
		\pi^{(l)} 
		&= 
		\argmin_{\pi \in \Pi_c(\mu, \nu^{(l)})} \int c \,d\pi + D_{\rm KL}(\pi, \mu \otimes \tilde\nu) \\
		&= \argmin_{\pi \in \Pi_c(\mu, \nu^{(l)})} \int c \,d\pi + D_{\rm KL}(\pi, \mu \otimes \nu^{(l)}) + D_{\rm KL}(\nu^{(l)}, \tilde\nu) \\
		&= \argmin_{\pi \in \Pi_c(\mu, \nu^{(l)})} \int c \,d\pi + D_{\rm KL}(\pi, \mu \otimes \nu^{(l)}),
		\end{align*}
		where the last equality holds as the value of $D_{\rm KL}(\nu^{(l)}, \tilde\nu)$ does not depend on $\pi$.

		Thus, by Theorem \ref{thm:qual_stab}, we obtain that $\pi^{(k+\frac{1}{2})}$ converges in $W_p$ to $\pi^*$ for $k \rightarrow \infty$, $k \in \mathbb{N}$. 
		Since, as consequence of \eqref{eq:pythagorean_thm}, $D_{\rm KL}(\pi^{(k)}, \pi^{(k-\frac{1}{2})})$ converges to zero, $\pi^{(k)}$ converges weakly to $\pi^*$ as well for $k\in \mathbb{N}$. Finally, since the $p$-th moments of $\pi^{(k)}$ converge (by convergence of the moments of the marginals as shown above), $\pi^{(k)}$ also converges in $W_p$ to $\pi^*$, which completes the proof.\qedhere

	\end{proof}
\end{proposition}

\subsection{Linear convergence for bounded cost functions}
For the main result of this section, we require attainment of dual optimizers and adapted versions of Schrödinger equations.
We first define the adapted Schrödinger equations, \eqref{eq:sinkhorn_step_causal_dual} and \eqref{eq:sinkhorn_step_normal}, and illustrate their relation to the dual Sinkhorn step from Lemma \ref{lem:proj_dual}.
We recall that the supremum norm $\| \cdot \|_\infty$ of a function $f \colon \X \times \Y \to \mathbb R$ is given by $\|f\|_\infty := \sup_{(x,y) \in \X \times \Y} |f(x,y)|$.
\begin{lemma}[Dual Sinkhorn steps]
	\label{lem:sinkhorn_step_bounds}
	Let $g \in \mathcal{S}_{\Y, \nu}$ and we write $g = \sum_{t=1}^N g_t$, $g_t \in \mathcal{A}_{\Y, \nu}$. We define the causal Sinkhorn step $f = \sum_{t=1}^N f_t \in \mathcal{S}_{\X, \mu}$ recursively by setting $c_N = c$,
	\begin{equation}
	\label{eq:sinkhorn_step_causal_dual}
	\begin{aligned}
	a_{t} &= - \log \int \exp(g_{t} - c_{t}) \,d\nu_{t}^{Y_{1:t-1}},&& t=1, \dots, N, \\
	f_{t+1} &= a_{t+1} - c_{t}, ~\text{ where } c_{t} = \int a_{t+1} \,d\mu_{t+1}^{X_{1:t}}, && t=1, \dots, N-1,
	\end{aligned}
	\end{equation}
	and $f_1 = a_1$. Then, $f$ satisfies
	\begin{align*}
	f &\in \argmax_{s \in \mathcal{S}_{\X, \mu}} \int s + g - \exp(s+g-c) + 1 \,dP, \\
	\|f_t\|_\infty &\leq 2 \Big( \sum_{s=t+1}^N \|g_s\|_\infty + \|c\|_\infty\Big), ~~~t = 2, \dots, N,\\
	f_1(x_1) &\in \Big[\int f_1 \,d\mu_1 - 2 \Big( \sum_{s=2}^N \|g_s\|_\infty + \|c\|_\infty\Big), \int f_1 \,d\mu_1 + 2 \Big( \sum_{s=2}^N \|g_s\|_\infty + \|c\|_\infty\Big)\Big],\, x_1 \in \X_1.
	\end{align*}
	The converse anticausal step is defined analogously 
	by inverting the roles of the coordinates $X$ and $Y$ and of the functions $f$ and $g$, and satisfies the same bounds. Finally, given $f \in \mathcal{S}_{\X, \mu}$, the converse normal step is given by defining $g \in \mathcal{S}_\Y$ via 
	\begin{equation}
	\label{eq:sinkhorn_step_normal}
	g = - \log \int \exp(f - c) \,d\mu,
	\end{equation}
	which satisfies, with decomposition $g = \sum_{t=1}^N g_t$ as in Remark \ref{rem:decomp}, 
	\begin{align*}
	g &\in \argmax_{s \in \mathcal{S}_\Y} \int f + s - \exp(f+s-c) + 1 \,dP, \\
	\|g_t\|_\infty &\leq 2 \Big( \sum_{s=t+1}^N \|f_s\|_\infty + \|c\|_\infty \Big), ~ t=2, \dots, N,\\
	g_1(y_1) &\in \Big[\int g_1 \,d\nu_1 - 2 \Big( \sum_{s=2}^N \|f_s\|_\infty + \|c\|_\infty\Big), \int g_1 \,d\nu_1 + 2 \Big( \sum_{s=2}^N \|f_s\|_\infty + \|c\|_\infty\Big)\Big], \, y_1 \in \Y_1.
	\end{align*}
	\begin{proof}
		First, we show the maximizing property of $f$. By weak duality (i.e., Fenchel-Young inequality  applied to the function $y \log y$, $xy \leq y \log y + \exp(x-1)$ for $x,y \in \mathbb R^+$)  we have that
		\[
		\sup_{s \in \mathcal{S}_{\X, \mu}} \int s - \exp(s + g - c) +1 \,dP \leq \inf_{\pi \in \Pi_c(\mu, *)} \int g+c\,d\pi + D_{\rm KL}(\pi, P).
		\]
		Further, one quickly sees that duality is attained whenever there exists a $\pi \in \Pi(\mu, *)$ such that $\frac{d\pi}{dP} = \exp(s + g - c)$ for some $s \in \mathcal{S}_{\X, \mu}$, since then, for any $\pi \in \Pi_c(\mu, *)$
		\[
		\int s - \exp(s + g - c) + 1 \,dP = \int s \,dP = \int s + (g-c) - (g-c) \,d\pi = \int g-c \, d\pi + D_{\rm KL}(\pi, P)
		\]
		where a construction for such a function $s$ is given in Lemma \ref{lem:proj_dual}. This construction coincides with the definition of $f$ here, and thus $f$ is a maximizer.\footnote{Note the simplification that since $g_{s}$ for $s < t$ does not depend on $X_{t:T}$, it does not influence $f_t$ and can thus be disregarded in the definition of $a_t$. Further, note that the assumption in Lemma \ref{lem:proj_dual} such that $\exp(g-c)$ is a density can be made without loss of generality by normalization.}
		
		Regarding the bounds for $f$, we introduce some notation: for some function $F : \X \times \Y \rightarrow \mathbb{R}$, we denote by 
		\[
		\Var(F, X_t) := \sup\{ F(x, y) - F(\tilde{x}, y) : y \in \Y, x, \tilde{x} \in \X \text{ s.t.~} x_s = \tilde{x}_s \text{ for all } s \neq t\}
		\]
		its maximum variation in the $t$-th variable of $\X$ (and similarly we define $\Var(F, Y_t)$). Then one finds $\Var(a_t, X_t) \leq \Var(c_t, X_t)$ by the triangle inequality and since $g_t$ is constant in $X_t$. Since for $t \geq 2$, $f_t$ is a centered version (in the variable $X_t$, given all other variables) of $a_t$, it thus holds 
		\begin{equation}
		\label{eq:sink_bound_eq1}
		\|f_t\|_\infty \leq \Var(a_t, X_t) \leq \Var(c_t, X_t) \leq 2 \|c_t\|_\infty \leq 2 \|a_{t+1}\|_\infty.
		\end{equation}
		Further, a straightforward calculation reveals $\|a_t\|_\infty \leq \|g_t\|_\infty + \|a_{t+1}\|_\infty$, which recursively and combined with \eqref{eq:sink_bound_eq1} leads to the stated bound for $f_t$ for $t \geq 2$. The case $t=1$ works analogously by noting that $f_1$ takes values in the range $\int f_1 \,d\mu_1 \pm \Var(f_1, X_1)$.
		
		Regarding $g$, the maximizing property in this case is the usual Schrödinger equation, and thus clear. The bounds follow completely analogously to the case for $f$, noting that $\Var(g_t, Y_t) \leq \Var(c, Y_t) + \sum_{s=t+1}^N \Var(f_s, Y_t)$.
	\end{proof}
\end{lemma}

\begin{lemma}[Dual optimizers are attained]
	Let $c$ be measurable and bounded, then
	\begin{align*}
	\sup_{h \in \mathcal{H}_\bullet} 
	\int h - \exp(h - c) \,dP + 1
	= 
	\max_{h \in \mathcal{H}_\bullet} 
	\int h - \exp(h-c) \,dP + 1.
	\end{align*}
	\begin{proof}
		Take a maximizing sequence $(\hat{h}^k)_{k \in \mathbb{N}}$. Define $h^k$ by applying to $\hat{h}^k$ the Sinkhorn step (alternating causal and anticausal, respectively causal and normal) $N+1$-times, as defined in Lemma \ref{lem:sinkhorn_step_bounds}. Then, $h^k$ is still a maximizing sequence since Lemma \ref{lem:sinkhorn_step_bounds} shows that the value of $\int h - \exp(h - c) \, dP$ can only increase by applying the Sinkhorn steps, and further $h^k$ is uniformly bounded by a constant only depending on $N$ and $\|c\|_\infty$. In fact, each component of $h^k = \sum_{t=1}^N f_t^k + g_t^k$ is uniformly bounded, since we can assume without loss of generality (for example by translating $f_1^k$ by a constant $b \in \mathbb R$ and $g_1^k$ by $-b$) that $\int f_1^k \,d\mu_1 = \int g_1^k \,d\nu_1$, which are both bounded above and below by $\|c\|_\infty$. We can apply Koml\'{o}s' Lemma \cite[Theorem 1a]{komlos1967generalization} which yields a subsequence (which we relabel by the initial sequence) such that
		\[
		\tilde{f}_t^k := \frac{1}{k} \sum_{i=1}^k f_t^k
		\text{ and } 
		\tilde{g}_t^k := \frac{1}{k} \sum_{i=1}^k g_t^k
		\]
		converge $P$-almost surely to some element $f_t^*, g_t^* \in L^1(P)$, respectively, and the limits clearly satisfy the same uniform bounds as the sequence. 
		Write $\tilde h^k := \sum_{t = 1}^N \tilde f_t^k + \tilde g_t^k$.
		By uniform boundedness, this convergence also holds in $L^1(P)$ and thus $f_t^* \in \mathcal{A}_{\X, \mu, t}, g_t^* \in \mathcal{A}_{\Y, \nu, t}$ and thus $h^* = \sum_{t=1}^N f_t^* + g_t^* \in \mathcal{H}_\bullet$.
		The next computation reveals optimality of $h^\ast$:
		\begin{align*}
		\int h^* - \exp(h^* - c) \,dP &\geq \limsup_{k\rightarrow \infty} \int \tilde{h}^k - \exp(\tilde{h}^k - c) \,dP \\ 
		&\geq \limsup_{k\rightarrow \infty} \frac{1}{k} \sum_{i=1}^k \int h^i - \exp(h^i - c) \,dP \\
		&= \sup_{h \in \mathcal{H}_\bullet} \int h - \exp(h - c) \,dP,
		\end{align*}
		where we used Fatou's lemma  and the concavity of $x \mapsto x - \exp(x - a)$ for an arbitrary $a \in \mathbb{R}$, which is uniformly upper bounded when $a \le \|c\|_\infty$.
	\end{proof}
\end{lemma}

\begin{lemma}[Schrödinger equations]\label{lem:schrodinger}
	Let $c$ be bounded, then any dual maximizer 
	\[
	h^* \in \argmax_{h \in \mathcal{H}_\bullet} \int h - \exp(h-c) \,dP + 1
	\]
	with $h^* = f^*+g^* = \sum_{t=1}^N f^*_t + g^*_t$
	satisfies the adapted versions of the Schrödinger equations, i.e., $f^*$ satisfies \eqref{eq:sinkhorn_step_causal_dual} and, in case of bicausal OT, $g^*$ satisfies the respective anticausal analogue of \eqref{eq:sinkhorn_step_causal_dual}, and in case of causal OT, $g^*$ satisfies \eqref{eq:sinkhorn_step_normal}. Further, using the normalization $\int f_1 \,d\mu_1 = \int g_1 \,d\nu_1$, there exists a constant $C_N$ only depending on $N$ such that
	\[\max\{\|h^*\|_\infty, \|f_1^*\|_\infty, \|f_2^*\|_\infty, \dots, \|f_N^*\|_\infty, \|g_1^*\|_\infty, \dots, \|g_N^*\|_\infty\} \leq C_N \, \|c\|_\infty.\]
\end{lemma}
\begin{proof}
	Let $a_t, c_t$, $t = 1, \ldots, N$ be given as in \eqref{eq:sinkhorn_step_causal_dual}.
	
	We have to show that $f^\ast_t$ satisfies, for $t = 1,\ldots,N$,
	\begin{equation}
	\label{eq:first_order_cond_f_schrodinger}
	f_t^\ast = 
	\begin{cases}
	a_1 & t = 1,\\
	a_t - c_{t-1} & \text{else}.
	\end{cases}
	\end{equation}
	Note that if \eqref{eq:first_order_cond_f_schrodinger} holds for $t < N$ and $s = t+1,\ldots, N$, then $P$-almost surely
	\begin{equation}
	\label{eq:first_order_cond_remark}
	\int \exp \Big( \sum_{s = t+1}^N f_s^\ast + g_s^\ast - c_s + c_{s - 1} \Big) \, dP_{t+1:N, t+1:N}^{X_{1:t}, Y_{1:t}} = 1.
	\end{equation}
	Next, we differentiate under the integral sign, which is justified by standard results, e.g., \cite[Theorem 2.27]{folland1999real}), in order to obtain first order optimality conditions that $h^*$ has to satisfy:
	for $t = 2,\ldots,N$ consider the first order conditions for the mapping
	\begin{equation}
	\label{eq:first_first_order_cond_mapping}
	\alpha \mapsto 
	\int 
	h^* + \alpha \tilde{f}_t - \exp(h^* + \alpha \tilde{f}_t - c) 
	\,dP
	\end{equation}
	and the test function $\tilde f_t \in \mathcal H_{c}$ given by
	\[
	\tilde{f}_t(x_{1:t}, y_{1:t-1}) = z_t(x_{1:t-1}, y_{1:t-1}) \underbrace{\Big(u_t(x_t) - \int u_t(\tilde{x}_t) \mu_{t}^{x_{1:t-1}}(d\tilde{x}_t)\Big)}_{=: U_t(x_t)},
	\]
	where $z_t$ and $u_t$ are continuous and bounded.
	Differentiation of \eqref{eq:first_first_order_cond_mapping} at $\alpha=0$ gives
	\begin{align*}
	0 
	&=
	\int -\tilde f_t + \tilde f_t \exp(h^\ast - c) \, dP =
	\int \tilde f_t \exp(h^\ast - c) \, dP
	\\
	&=
	\int z_t \exp \Big( \sum_{i = 1}^{t-1} f_i^\ast + g_i^\ast \Big) U_t \exp\Big( - c_{t-1} + \sum_{i = t}^N f_i^\ast + g_i^\ast - c_i + c_{i - 1} \Big)
	\, dP,
	\end{align*}
	where the first equality is due to optimality of $h^\ast$.
	As $z_t$ was arbitrary, this yields $P$-almost surely
	\begin{align*}
	\int U_t \exp\Big( - c_{t-1} + \sum_{i = t}^N f_i^\ast + g_i^\ast - c_i + c_{i - 1} \Big)
	\, dP_{t:N,t:N}^{X_{1:t-1}, Y_{1:t-1}} = 0,
	\end{align*}
	from which we derive, as $U_t$ was arbitrary, that $P$-almost surely
	\[
	K_t := \int \exp \Big( - c_{t-1} + \sum_{i = t}^N f_i^\ast + g_i^\ast - c_i + c_{i - 1} \Big) \, dP_{t+1:N,t:N}^{X_{1:t}, Y_{1:t-1}},
	\]
	is a measurable function of $x_{1:t-1}$ and $y_{1:t-1}$; in other words, $K_t$ is constant in $x_t$.
	
	We prove inductively the following claim:
	For $t = 1, \ldots, N$, we have \eqref{eq:first_order_cond_f_schrodinger}, and for $t = 1,\ldots, N -1$ we have \eqref{eq:first_order_cond_remark}.
	
	Let $t = N$, or $t = 2,\ldots, N-1$ such that $f_s^\ast = a_s - c_{s - 1}$ for $s = t + 1$.
	In both cases, we find thanks to \eqref{eq:first_order_cond_remark} that $P$-almost surely
	\[
	K_t = \exp( f_t^\ast ) \int \exp( g_t^\ast - c_t ) \, d\nu_{t}^{Y_{1:t-1}} = \exp( f_t^\ast - a_t).
	\]
	The next computation yields \eqref{eq:first_order_cond_f_schrodinger} for $t$, since
	\begin{align*}
	f_t^\ast &= f_t^\ast - \int f_t^\ast \, d\mu_t^{X_{1:t-1}} 
	\\
	&=
	\log(K_t) + a_t - \int \log(K_t) + a_t \, d\mu_t^{X_{1:t-1}} = a_t - c_{t-1}.
	\end{align*}
	Assume now that $f_t^\ast = a_t - c_{t - 1}$ for $t = 2,\ldots, N$, then \eqref{eq:first_order_cond_remark} yields that
	\[
	\int h^\ast - \exp(h^\ast - c) \, dP = \int f_1^\ast + g_1^\ast - \exp( f_1^\ast + g_1^\ast - c_1) \, d\mu_1 \otimes \nu_1.
	\]
	This is the same objective function as for normal entropic OT with cost function $c_1$, and thus the case $t=1$ of \eqref{eq:first_order_cond_f_schrodinger} is just the non-adapted Schrödinger equation. This completes the proof of \eqref{eq:first_order_cond_f_schrodinger}. In the bicausal case, the respective analogue for $g^*$ follows by symmetry. And for the causal case, \eqref{eq:sinkhorn_step_normal} for $g^*$ is quickly derived similar to the non-adapted Schrödinger equation, and thus omitted.

	The uniform bounds are satisfied by noting that the same bounds as given by Lemma \ref{lem:sinkhorn_step_bounds} are given for $f^*$ and $g^*$ and thus by working backwards in time, the claim follows.
\end{proof}

The next lemma is insightful in order to understand Sinkhorn iterations in their dual form.
Though the lemma is stated for causal projections, there are obvious analogues for classical and anticausal projections.
\begin{lemma}
	\label{lem:invarianceopti}
	Let $c, g \in L^\infty(\X \times \Y)$ and $f \in \mathcal{H}_c$, and $\pi^{(a)}, \pi^{(b)} \in \mathcal{P}(\X \times \Y)$ with densities
	\[
	\frac{d\pi^{(a)}}{d P} = \exp(g-c) ~ \text{ and } ~ \frac{d\pi^{(b)}}{d P} = \exp(g + f - c).
	\]
	Then
	\[
	\argmin_{\pi \in \Pi_c(\mu, *)} D_{\rm KL}(\pi, \pi^{(a)}) = \argmin_{\pi \in \Pi_c(\mu, *)} D_{\rm KL}(\pi, \pi^{(b)}).
	\]
\end{lemma}
\begin{proof}
	It suffices to show that $D_{\rm KL}(\pi, \pi^{(a)}) - D_{\rm KL}(\pi, \pi^{(b)})$ is constant for any $\pi \in \Pi_c(\mu, \ast)$ with
	$D_{\rm KL}(\pi, \pi^{(a)}) < \infty$ or $D_{\rm KL}(\pi, \pi^{(b)}) < \infty$.
	Since $f \in \mathcal H_c$, we have that $\int f \, d\pi$ is constant for all $\pi \in \Pi_c(\mu,\ast)$.
	Assume w.l.o.g.\ that the entropy of $\pi$ w.r.t.\ $\pi^{(a)}$ is finite, then we conclude with
	\begin{align*}
	D_{\rm KL}(\pi, \pi^{(a)}) - \int f \, d\pi 
	&=
	\int \log \Big( \frac{d\pi}{d\pi^{(a)}} \Big) - \log \Big( \frac{d\pi^{(b)}}{d\pi^{(a)}} \Big) \, d\pi
	\\
	&=
	\int \log \Big(\frac{d\pi}{d\pi^{(b)}}\Big) \, d\pi
	\\
	&=
	D_{\rm KL}(\pi, \pi^{(b)}). \qedhere
	\end{align*}
\end{proof}

In the following, we introduce notation for the dual steps for Sinkhorn's algorithm and a suitable normalization. Let $\pi^{(k)}, \pi^{(k+\frac{1}{2})}$ be the Sinkhorn iterations from Definition \ref{def:sinkhornsteps}. Further, denote by $\mathcal{H}_1 = \mathcal{S}_{\X, \mu}$ and $\mathcal{H}_2 \in \{\mathcal{S}_{\Y}, \mathcal{S}_{\Y, \nu}\}$ the corresponding dual sets (cf.~Section \ref{subsec:duality}) to the building blocks $\mathcal{Q}_1, \mathcal{Q}_2$ from Definition \ref{def:sinkhornsteps}.
We know by Lemma \ref{lem:proj_dual}, or equivalently Lemma \ref{lem:sinkhorn_step_bounds}, and its analogues that for $k \in \mathbb{N}$
\[
\frac{d\pi^{(k)}}{dP} = \exp(f^{(k)} + g^{(k)} - c)
\]
for some $f^{(k)} \in \mathcal{H}_1, g^{(k)} \in \mathcal{H}_2$. Further, by Lemma \ref{lem:invarianceopti}, the respective function $f^{(k)}$ only depends on the prior function $g^{(k-1)}$, but not on $f^{(k-1)}$, and similarly $g^{(k)}$ only depends on $f^{(k)}$, but not on $ g^{(k-1)}$. 
Thus, with this notation we have
\[
\frac{d\pi^{(k+\frac{1}{2})}}{dP} = \exp(f^{(k+1)} + g^{(k)} - c).
\]
We further denote by $f^*$ and $g^*$ the respective dual optimizers for $E_\bullet(\mu, \nu, c)$, and introduce the normalized versions
\begin{align*}
\tilde{f}^{(k)} &:= f^{(k)} - \int f^{(k)} \,dP, & \tilde{f}^* &= f^* - \int f^* \,dP \\
\tilde{g}^{(k)} &:= g^{(k)} + \int f^{(k)} \,dP, & \tilde{g}^* &= g^* + \int f^* \,dP.
\end{align*}


\begin{lemma}\,
	\label{lem:prop_sink_iter}
	The normalized Sinkhorn potentials satisfy the following.
	\begin{enumerate}[label=(\roman*)]
		\item \label{it:prop_sink_iter_1}
		There exists a constant $C_N$ only depending on $N$ such that \[
		\sup\{\|\tilde{f}^{(k)}\|_\infty \lor \|\tilde{g}^{(k)}\|_\infty : k \in \mathbb{N}\} \leq C_N \|c\|_\infty
		\]
		\item \label{it:prop_sink_iter_2}
		For any $k \in \mathbb N$, $\tilde f^{(k)} - \tilde f^\ast$ is orthogonal to $\tilde g^{(k)} - \tilde g^\ast$ in $L^2(P)$, i.e.,
		\[
		\|\tilde{f}^{(k)} + \tilde{g}^{(k)} - \tilde{f}^* - \tilde{g}^*\|_{L^2(P)}^2 = \|\tilde{f}^{(k)} - \tilde{f}^*\|_{L^2(P)}^2 + \|\tilde{g}^{(k)} - \tilde{g}^*\|_{L^2(P)}^2.
		\]
	\end{enumerate}
	\begin{proof}
		Part \ref{it:prop_sink_iter_1} follows from Lemma \ref{lem:sinkhorn_step_bounds} since one can inductively show that, the following bounds for $t = 2, \ldots, N$,
		\[
		\| g^{(k+1)}_t \|_\infty \le 2^{N + 1 - t} \| c \|_\infty
		\]
		imply the same bounds for $f^{(k)}_t$.
		We may argue mutatis mutandis for $g^{(k)}_t$.
		Note that
		\begin{align*}
		-\| c \|_\infty \le \int f^{(k)} + g^{(k)} \, dP \le \| c \|_\infty,
		\end{align*}
		where the first inequality follows by the maximizing property in Lemma \ref{lem:sinkhorn_step_bounds} and the last by weak duality.
		Finally, we have
		\[
		\left\| {f}_1^{(k)} - \int {f}^{(k)} \, dP \right\|_\infty \le 2^N \| c \|_\infty,
		\quad
		\left\| {g}_1^{(k)} + \int {f}^{(k)} \, dP \right\|_\infty
		\le (2^N + 1) \| c \|_\infty,
		\]
		which easily yields the claim.
		
		Regarding \ref{it:prop_sink_iter_2}, we have to show that
		\[
		\int (\tilde{f}^{(k)} - \tilde{f}^*) (\tilde{g}^{(k)} - \tilde{g}^*) \,dP = 0
		\]
		This can again be shown via backward induction. 
		Indeed, let 
		\[
		D_{f, t} := \tilde{f}^{(k)}_t - \tilde{f}^*_t, \quad
		D_{g, t} := \tilde{g}^{(k)}_t - \tilde{g}^*_t,
		\]
		where we use the same notation as in Lemma \ref{lem:schrodinger}.
		Note that, for $1 \le s \le t$, $D_{f,s}$ does not depend on $y_t$ and similarly $D_{g,s}$ does not depend on $x_t$, hence,
		\[
		\int D_{f,t} \Big( \sum_{s = 1} D_{g,s} \Big) \, dP 
		= 0 
		= \int \Big( \sum_{s = 1} D_{f,s} \Big) D_{g,t} \, dP.
		\]
		Then, for $t = 2,\ldots,N$, we get
		\begin{align*}
		&\int \Big(\sum_{s=1}^t D_{f, s}\Big) \Big( \sum_{s=1}^t D_{g, s}\Big) \,dP 
		\\
		&= \int \Big(\sum_{s=1}^{t-1} D_{f, s}\Big) D_{g, t} + D_{f, t} \Big(\sum_{s=1}^{t-1} D_{g, s}\Big) + D_{g, t} D_{f, t} + \Big(\sum_{s=1}^{t-1} D_{f, s}\Big) \Big(\sum_{s=1}^{t-1} D_{g, s}\Big) \,dP 
		\\
		&= \int \Big(\sum_{s=1}^{t-1} D_{f, s}\Big) \Big(\sum_{s=1}^{t-1} D_{g, s}\Big) \,dP.
		\end{align*}
		Therefore, we conclude with
		\[
		\int (\tilde{f}^{(k)} - \tilde{f}^*) (\tilde{g}^{(k)} - \tilde{g}^*) \,dP = \int (\tilde{f}^{(k)}_1 - \tilde{f}^*_1) (\tilde{g}^{(k)}_1 - \tilde{g}^*_1) \,dP,
		\]
		which vanishes by definition of $\tilde{f}_1^{(k)}$ and $\tilde{f}_1^*$.
	\end{proof}
\end{lemma}

To state the following theorem, denote by $E^{(k)}$ the objective value of Sinkhorn after the $k$-th step, i.e., 
\[
E^{(k)} := \int c \, d\pi^{(k)} + D_{\rm KL}(\pi^{(k)}, P) = \int f^{(k)} + g^{(k)} - \exp(f^{(k)} + g^{(k)} - c) \,dP +1,
\]
which obviously satisfies $E^{(k)} \leq E(\mu, \nu, c)$.
The following result is based on the idea of Carlier \cite{carlier2021linear} for the classical Sinkhorn's algorithm.
\begin{theorem}[Linear convergence of Sinkhorn's algorithm]
	\label{thm:sink_lin_conv}
	There exist constants $\rho \in (0, 1)$ and $C > 0$ only depending on $\|c\|_\infty$ and $N$ such that
	\begin{align*}
	\big(E_\bullet(\mu, \nu, c) - E^{(k)}\big) &\leq \rho^k \, \Big(E_\bullet(\mu, \nu, c) - E^{(0)}\big) \\
	\|\tilde{f}^{(k)} - \tilde{f}^*\|_{L^2(P)}^2   + \|\tilde{g}^{(k)} - \tilde{g}^*\|_{L^2(P)}^2 &\leq C \rho^k \, \Big(E_\bullet(\mu, \nu, c) - E^{(0)}\big)
	\end{align*}
	\begin{proof}
		We will use strong convexity of the exponential function, i.e., we have for $a, b \in [-\alpha, \alpha]$, and $\alpha > 0$
		\begin{equation}
		\label{eq:strong convexity}
		\exp(b) - \exp(a) \geq \exp(a) (b-a) + \frac{\exp(-\alpha)}{2} (b-a)^2.
		\end{equation}
		In our case, as all functions are bounded (cf.~Lemma \ref{lem:prop_sink_iter} \ref{it:prop_sink_iter_1}), we will apply this inequality with $\alpha = 2 C_N \|c\|_\infty$, and denote by $\sigma := \exp(-\alpha)$.
		
		To ease notation, write $E = E_\bullet(\mu, \nu, c)$. 
		Using \eqref{eq:strong convexity}, we obtain
		\begin{align}
		\begin{split}
		E^{(k)} - E &\geq \int (1-\exp(\tilde{f}^{(k)} + \tilde{g}^{(k)} - c)) (\tilde{f}^{(k)} + \tilde{g}^{(k)} - \tilde{f}^* - \tilde{g}^*) \,dP  
		\\
		&\phantom{=}+ \frac{\sigma}{2} \|\tilde{f}^{(k)} + \tilde{g}^{(k)} - \tilde{f}^* - \tilde{g}^*\|_{L^2{P}}^2 \\
		&= \int (1-\exp(\tilde{f}^{(k)} + \tilde{g}^{(k)} - c)) (\tilde{f}^{(k)} - \tilde{f}^*) \,dP \\
		&\phantom{=}+ \int (1-\exp(\tilde{f}^{(k)} + \tilde{g}^{(k)} - c)) (\tilde{g}^{(k)} - \tilde{g}^*) \,dP \\
		&\phantom{=}+ \frac{\sigma}{2} (\|\tilde{f}^{(k)} - \tilde{f}^*\|_{L^2(P)}^2 + \|\tilde{g}^{(k)} - \tilde{g}^*\|_{L^2(P)}^2).
		\end{split}
		\label{eq:strong_conv}
		\end{align}
		where the equality uses Lemma \ref{lem:prop_sink_iter} \ref{it:prop_sink_iter_2}. 
		First, note that 
		\begin{equation}
		\label{eq:vanishing term}
		\int (1-\exp(\tilde{f}^{(k)} + \tilde{g}^{(k)} - c)) (\tilde{g}^{(k)} - \tilde{g}^*) \,dP = 0,
		\end{equation}
		since $\tilde{g}^{(k)}$ is chosen in a way so that $\exp(\tilde{f}^{(k)} + \tilde{g}^{(k)} - c)$ is the density of a measure contained in $\Pi_{ac}(*, \nu)$ (respectively $\Pi(*, \nu)$) w.r.t.\ $P$. 
		
		Next, we treat the term $\int (1-\exp(\tilde{f}^{(k)} + \tilde{g}^{(k)} - c)) (\tilde{f}^{(k)} - \tilde{f}^*) \,dP$ in the above.
		Let 
		\[
		A_{k,k} := 1-\exp(\tilde{f}^{(k)} + \tilde{g}^{(k)} - c),
		\quad 
		A_{k+1, k} := 1-\exp(\tilde{f}^{(k+1)} + \tilde{g}^{(k)} - c).
		\]
		Again, as $\exp(\tilde{f}^{(k+1)} + \tilde{g}^{(k)} - c)$ is up to a normalizing factor $\beta > 0$ the density of $\pi^{(k+\frac{1}{2})} \in \Pi_c(\mu, \ast)$ w.r.t.\ $P$, we find that
		\[
		\int A_{k+1, k} (\tilde{f}^{(k)} - \tilde{f}^*) \,dP = \int (\tilde{f}^{(k)} - \tilde{f}^*) \,dP - \beta \int (\tilde{f}^{(k)} - \tilde{f}^*) \,d\pi^{(k+\frac{1}{2})}
		\]
		vanishes, because both terms on the right-hand side are zero by definition of $\tilde{f}^{(k)}$ and $\tilde{f}^*$. 
		Hence, we can apply Young's inequality in order to derive
		\begin{align}
		\label{eq:bound by young}
		\begin{split}
		\int (1-\exp(\tilde{f}^{(k)} + \tilde{g}^{(k)} - c)) (\tilde{f}^{(k)} - \tilde{f}^*) \,dP 
		= \int (A_{k, k} - A_{k+1, k}) (\tilde{f}^{(k)} - \tilde{f}^*) \,dP \\
		\geq - \frac{1}{2\sigma} \|A_{k, k} - A_{k+1, k}\|_{L^2(P)}^2 - \frac{\sigma}{2} \|\tilde{f}^{(k)} - \tilde{f}^*\|_{L^2(P)}^2.
		\end{split}
		\end{align}
		By combining \eqref{eq:strong_conv}, \eqref{eq:vanishing term} and \eqref{eq:bound by young}, we obtain
		\begin{align*}
		E - E^{(k)} 
		&\leq \frac{1}{2\sigma} \| A_{k,k} - A_{k+1,k} \|_{L^2(P)}^2 + \frac{\sigma}{2} \| \tilde g^{(k)} - \tilde g^\ast \|_{L^2(P)}^2
		\\
		&\le \frac{1}{2\sigma} \|A_{k, k} - A_{k+1, k}\|_{L^2(P)}^2 
		\\
		&\leq \frac{ 
			1}{2\sigma^3} \|\tilde{f}^{(k+1)} - \tilde{f}^{(k)}\|_{L^2(P)}^2,
		\end{align*}
		where the last inequality holds due to the Lipschitz property of the exponential restricted to $[-\alpha,\alpha]$ with constant $1 / \sigma$.
		Finally, completely analogously to \eqref{eq:strong_conv} we get (where we adopt the notation $E^{(k+\frac{1}{2})} = \int \tilde{f}^{(k+1)} + \tilde{g}^{(k)} - \exp(\tilde{f}^{(k+1)} + \tilde{g}^{(k)} -c) \,dP$)
		\begin{align*}
		E^{(k+1)} - E^{(k)} 
		&= 
		\big(E^{(k+1)} - E^{(k+\frac{1}{2})}\big) 
		+\big(E^{(k+\frac{1}{2})} - E^{(k)}\big) 
		\\ 
		&\geq \frac{\sigma}{2}
		\Big( \|\tilde{f}^{(k+1)} - \tilde{f}^{(k)}\|_{L^2(P)}^2 + \| \tilde{g}^{(k+1)} - \tilde{g}^{(k)}\|_{L^2(P)}^2
		\Big),
		\end{align*}
		hence, the value $E^{(k)}$ is increasing in $k$ and
		\[
		E - E^{(k)} 
		\leq \frac{1}{2\sigma^3} 
		\|\tilde{f}^{(k+1)} - \tilde{f}^{(k)}\|_{L^2(P)}^2 
		\leq \frac{1}{\sigma^4} \big( E^{(k+1)} - E^{(k)} \big).
		\] 
		Then choosing $\rho = 1 - \sigma^4$ yields
		\[
		E - E^{(k + 1)} \le \rho \big( E - E^{(k)} \big),
		\]
		from where we easily conclude the first claim.
		The second claim follows by
		\[
		E - E^{(k)} \geq \frac{\sigma}{2} \big(\|\tilde{f}^{(k)} - \tilde{f}^*\|_{L^2(P)}^2 + \|\tilde{g}^{(k)} - \tilde{g}^*\|_{L^2(P)}^2\big),
		\]
		which can be shown analogously to \eqref{eq:strong_conv} and the subsequent line of arguments.
	\end{proof}
\end{theorem}

\begin{remark}[Sinkhorn's algorithm in backward induction]\label{rem:sink_back}
	We shortly discuss the use of Sinkhorn's algorithm to (approximately) solve the optimal transport problems in the backward induction of bicausal transport, as done in \cite{pichler2021nested}.
	There are interesting connections of this utilization of Sinkhorn's algorithm compared to the adapted version of Sinkhorn's algorithm studied in this paper. 
	First, both are (part of) algorithms to solve the entropic transport problem $E_{bc}(\mu, \nu, c)$.
	
	Concerning the algorithm proposed in the present paper, one may regard the Sinkhorn steps as the ``outer algorithm'', which operate on joint couplings for the whole path space. 
	Each step within the Sinkhorn algorithm, that is the projection onto $\Pi_c(\mu, *)$ or $\Pi_{ac}(*, \nu)$, respectively, are calculated using a backward induction, see Lemma \ref{lem:proj_dual}. 
	Thus, the backward induction can be seen as the ``inner algorithm''.
	
	In contrast, concerning the algorithm proposed in  \cite{pichler2021nested}, one may regard the backward induction therein as the ``outer algorithm'', whereas Sinkhorn's algorithm is used to solve the respective steps of the backward induction. 
	Thus, Sinkhorn's algorithm takes the role of the ``inner algorithm''.
	
	There are theoretical arguments for both utilizations of Sinkhorn's algorithm. First, we showed that the adapted version of Sinkhorn's algorithm converges, even linearly, while the backward induction steps (the ``inner algorithm'') can -- at least for discrete spaces -- be calculated precisely. Thus, the only errors resulting from the adapted version of Sinkhorn's algorithm are happening in the outer loop.
	If convergence has not been reached yet, with this structure the algorithm can simply continue iterating in order to achieve convergence.
	
	On the other hand, using Sinkhorn's algorithm to calculate the steps of the backward induction as in \cite{pichler2021nested} introduces an error in each step. 
	Since the value of each time step $t+1$ is used as an input to the next step $t$, this error propagates.
	Even though this error can be controlled, because entropic optimal transport is stable with respect to perturbations in the cost function, see, e.g., \cite{eckstein2021quantitative, keriven2022entropic}, one may intuitively wish to avoid such propagating errors.
	
	It is however also plausible that using Sinkhorn's algorithm to calculate the steps of the backward induction as in \cite{pichler2021nested} improves the speed of convergence compared to the adapted version of Sinkhorn's algorithm. Indeed, with the adapted version of Sinkhorn's algorithm, to calculate the transition kernels for time $t$, one utilizes the current sub-optimal solution of the later time steps $t+1, \dots, N$, since these are taken from the same iteration step. However, using backward induction as the outer algorithm, the later time steps $t+1, \dots, N$ have already reached their convergence criteria and can thus be regarded as more accurate inputs to the current step $t$.
\end{remark}

\subsection{Numerical example: comparison of algorithms} \label{subsec:numalgo}
To give a basic idea of the practical applicability of the introduced version of Sinkhorn's algorithm, this section reports numerical results, where the focus is on both runtime and accuracy. The code to reproduce the experiments is available at \url{https://github.com/stephaneckstein/aotnumerics}. 

The marginals $\mu$ and $\nu$ will be randomly generated Markovian trees on $\mathbb{R}^N$, $N = 3$, with width $W$ and number of branches $n_b$ (cf.~\cite[Chapter 5]{pichler2021nested}), meaning that the transition kernels $\mu_t^{x_{t-1}}$ will be supported on at most $n_B$ integer points in the range $[x_{t-1} - W, x_{t-1} + W]$. We choose two cost functions $c_1$ and $c_2$ given by
\begin{align*}
c_1(x, y) = \sum_{t=1}^N (x_t - y_t)^2 / (2W)^2,\\
c_2(x, y) = \sum_{t=1}^N \sin(x_t y_t) + |x_t-y_t|/W,
\end{align*}
where the normalization by $W$ is used so that the supremum norm of the cost function on the support of the considered measures is around one, and thus the regularization parameter $\varepsilon$ can be reasonably chosen within the same range for both cost functions.

\begin{table}
\caption{Comparison of different algorithms to compute causal and bicausal optimal transport problems. The table reports runtimes in seconds (and in brackets relative errors for regularized problems) to calculate $V_\bullet(\mu, \nu, c_i)$ (respectively $E_\bullet^\varepsilon(\mu, \nu, c_i)$). The marginals are randomly chosen Markovian trees with three time steps and width $n_b$. The reported relative errors are $(\int c_i d\hat{\pi} - \int c_i d\pi^*)/(\int c_i \,d\mu \otimes \nu - \int c_i \,d\pi^*)$, where $\hat\pi$ is the approximate optimizer for the regularized problem, and $\pi^*$ is the real optimizer.  All values are averaged over 10 runs with different randomly generated marginals. The reported runtimes arise from running the programs on a single $3.1$GHz CPU. Regarding algorithms, LP implements the formulation of Lemma \ref{lem:discretization}. DPP refers to backward induction, either exact (LP) or using Sinkhorn's algorithm ($\varepsilon=0.01$) to solve the OT problems within the backward induction. Sinkhorn refers to the algorithm introduced in Section \ref{sec:sinkhorn}. LP for $n_b=50$ did not converge in under 1000 seconds.}

	\begin{tabular}{l l l l l l l}
	\toprule
	 &  & LP & 
	 \begin{tabular}{@{}l@{}}DPP \\ (LP)\end{tabular}  & \begin{tabular}{@{}l@{}}DPP \\ $(\varepsilon=0.01$)\end{tabular}  & \begin{tabular}{@{}l@{}}Sinkhorn \\ ($\varepsilon=0.1$)\end{tabular} & \begin{tabular}{@{}l@{}}Sinkhorn \\ ($\varepsilon=0.01$)\end{tabular}    \\\midrule
  	& \multicolumn{6}{c}{cost function $c_1$} \\
  	$\bullet$ & $n_b$ &  &  &  &  &  \\
	$c$ & 10 & 0.97 & - & - & 0.08 (12.7\%) & 0.74 (0.48\%)   \\
	$c$ & 25 & 107.10 & - & - & 0.47 (14.43\%) & 4.31 (1.13\%) \\
	$c$ & 50 & - & - & - & 2.34 (-) & 19.37 (-) \\
 	$bc$ & 10 & 1.61 & 0.60 & 1.08 (0.81\%) & 0.08 (13.82\%)  & 0.97 (0.46\%) \\
	$bc$ & 25 & 205.10 & 24.09 & 11.55 (1.51\%) & 0.50 (16.25\%) & 4.78 (1.24\%) \\
	$bc$ & 50 & - & 308.55 & 79.72 (1.73\%) & 2.21 (16.10\%) & 18.60 (1.56\%) \\
    \vspace{0.2cm} \\ 
	& \multicolumn{6}{c}{cost function $c_2$} \\
  	$\bullet$ & $n_b$ &  &  &  &  &  \\
	$c$ & 10 & 1.08 & - & - & 0.22 (0.99\%) & 3.02 (0.07\%)  \\
	$c$ & 25 & 106.86 & - & - & 1.49 (1.70\%) & 18.13 (0.08\%) \\
	$c$ & 50 & - & - & - & 9.12 (-) & 105.01 (-) \\
 	$bc$ & 10 & 1.55 & 0.60 & 2.44 (0.06\%) & 0.27 (0.96\%) & 4.59 (0.10\%)  \\ 
	$bc$ & 25 & 197.23 & 21.61 & 23.86 (0.03\%) & 1.47 (1.87\%) & 22.09 (0.06\%) \\
	$bc$ & 50 & - & 315.70 & 108.97 (0.02\%) & 8.74 (2.37\%) & 105.86 (0.08\%) \\
\vspace{0.2cm} \\\bottomrule
\end{tabular}
\label{table:t1}
\end{table}

\begin{table}
\caption{Runtimes and relative errors (in brackets) for calculation of $V_{bc}(\mu, \nu, c_2)$. Reported values are averages over 10 (for $n_b=10, 25, 50$) respectively 3 (for $n_b=75, 100$) sample runs.}
	\begin{tabular}{l l l l l l}
	\toprule
  	$n_b$ & 10 & 25 & 50 & 75 & 100 \\
	DPP (LP) & 0.62 & 21.61 & 315.70 & 1285.48 & 3784.37   \\ 
	Sinkhorn ($\varepsilon = 0.01$) & 4.59 (0.10\%) & 22.09 (0.06\%) & 105.86 (0.08\%) & 230.40 (0.07\%) & 500.14 (0.04\%) \\\bottomrule
    \end{tabular}
    \label{table:t2}
\end{table}
The results are reported in Table \ref{table:t1}. Therein, we compare different algorithms, both exact (LP as in Lemma \ref{lem:discretization}, and backward induction for the bicausal problem) and approximate (causal and bicausal version of Sinkhorn's algorithm, and backward induction using Sinkhorn's algorithm as the inner solver). Linear programs were implemented using Gurobi \cite{gurobi}. Sinkhorn's algorithm for the backward induction was implemented using the PythonOT package \cite{flamary2021pot}. Causal and bicausal versions of Sinkhorn's algorithm as presented in this section were implemented from scratch in Python.

Table \ref{table:t1} shows that the runtime of the direct LP implementation using Lemma \ref{lem:discretization} increases quickly for increasing $n_b$, but we emphasize that it is the only method to produce exact values for the causal problem. Causal and bicausal version of Sinkhorn's algorithm give accurate results (for $\varepsilon=0.01$) in comparable times to methods based on backward induction for bicausal problems, and noticably also much faster times than the direct LP implementation for the causal problem. Further, using $\varepsilon = 0.1$ yields quick approximate solutions, even for large values of $n_b$. 

Comparing the columns DPP ($\varepsilon = 0.01$) and Sinkhorn ($\varepsilon = 0.01$), we note that the same theoretical problems are implemented, and the differences in relative error results purely from numerical inaccuracies and differing stopping criteria, cf.~Remark \ref{rem:sink_back}. We emphasize that both methods start to be significantly faster than the exact DPP implementation for $n_b=50$. We give additional values showing the extrapolation of this trend in Table \ref{table:t2}. Therein, we see that for even larger values of $n_b$, the Sinkhorn algorithm yields increasingly larger benefits in runtime compared to the exact backward induction method. Thus, Sinkhorn's algorithm has particular benefits if the actual optimal transport problems occurring in the backward induction are large ($n_b=50, 75, 100$). If however all transition kernels are supported on a small number of points ($n_b=10, 25$), there are no benefits in runtime compared to an exact backward induction using linear programming.

To conclude, numerically the introduced versions of Sinkhorn's algorithm give comparable results to state of the art methods based on backward induction for the considered bicausal problems, while having the benefit of also being applicable to causal problems. In both cases (causal and bicausal), the runtime of Sinkhorn's algorithm is drastically faster than a direct LP implementation even for moderately sized problems.

\begin{acks}[Acknowledgments]
	The authors are grateful to two anonymous referees for their insightful remarks and constructive feedback. SE thanks the Erwin Schr\"{o}dinger Institute for Mathematics and Physics, University of Vienna, where major parts of this project were completed.
\end{acks}

\appendix
\section{Topological aspects of AOT}


%
%

\begin{lemma}
	\label{lem:refined topology}
	Let $\X$ be a Polish space with topology $\tau^\X$, $(\mathcal Y_n)_{n \in \mathbb N}$ be a sequence of Polish spaces,
	and $(K_n)_{n \in \mathbb N}$ be a family of measurable kernels $K_n \colon \mathcal X \to \mathcal P(\mathcal Y_n)$.
	Then there exists a Polish topology $\tau \supseteq \tau^\mathcal X$ on $\X$ such that the Borel-$\sigma$-algebras generated by $\tau$ and $\tau^\mathcal X$ coincide, and, for $n \in \mathbb N$,
	\begin{align}
	\label{eq:lem_continuity of kernels}
	\begin{split}
	(\mathcal X, \tau) &\to \mathcal P(\mathcal Y_n), \\
	x &\mapsto K_n(x)
	\end{split}
	\end{align}
	is continuous.
	In particular, we have that
	\begin{align*}
	\left\{ f(x) = \int g(x,y) \, K_n(x,dy) \colon g \in C_b(\mathcal X \times \mathcal Y_n) \right\} 
	\subseteq
	C_b((\mathcal X, \tau)).
	\end{align*}
\end{lemma}

\begin{proof}
	The map $K \colon \mathcal X \to \prod_{n \in \mathbb N} \mathcal P(\mathcal Y_n)$ where $K(x) = (K_n(x))_{n \in \mathbb N}$ is measurable and $\prod_{n \in \mathbb N} \mathcal P(\mathcal Y_n)$ is Polish as a countable product.
	Hence, by \cite[Theorem 13.11]{kechris2012classical} there exists a Polish topology $\tau$ on $\mathcal X$ such that 
	\[
	K \colon (\mathcal X, \tau) \to \prod_{n \in \mathbb N} \mathcal P(\mathcal Y_n),
	\]
	is continuous, $\tau \supseteq \tau^\mathcal X$, and the Borel sets of $\tau$ and $\tau^\mathcal X$ coincide.
	Clearly, this is equivalent to \eqref{eq:lem_continuity of kernels}.
	Finally, note that, for fixed $g \in C_b(\mathcal X \times \mathcal Y_n)$, the map $p \mapsto \int g(x,y) \, p(dx,dy)$ is continuous on $\mathcal P(\mathcal X \times \mathcal Y_n)$.
	Consequently, $G(x,p) := \int g(x,y) \, p(dy)$ is continuous on $\mathcal X \times \mathcal P(\mathcal Y_n)$.
	Write $\iota_n \colon (\mathcal X,\tau) \to \X \times \mathcal P(\mathcal Y_n)$ for the continuous map $\iota_n(x) = (x,K_n(x))$.
	We can decompose
	\[
	f(x) := \int g(x,y) \, K_n(x,dy)
	\]
	into $f = G \circ \iota_n$, which shows $f \in C_b((\mathcal X,\tau))$.
\end{proof}

    \begin{lemma}
    \label{lem:causality_char}
        Let $\mu \in \mathcal P(\X)$, $\nu \in \mathcal P(Y)$, and $\pi \in \Pi(\mu,\nu)$.
        Then $\pi \in \Pi_c(\mu,\nu)$ if and only if we have, for all $t \in \{2,\ldots,N\}$,
        \[
            \int f_t \, d\pi = 0 \quad \forall f_t \in \tilde{\mathcal A}_{\X,\mu,t}.
        \]
    \end{lemma}
    
    \begin{proof}
        Since $\pi \in \Pi(\mu,\nu)$ we have by Definition \ref{def:causality} that $\pi$ is causal if and only if, for all $ t \in \{2,\ldots,N\}$,
        \[
            \pi_{1:t,1:t-1} = \pi_{1:t-1,1:t-1} \otimes \mu^{X_{1:t-1}}_{t}.
        \]
        As continuous functions are point-separating on $\mathcal P(\X_{1:t+1} \times Y_{1:t})$, see for example \cite[Chapter 3, Theorem 4.5]{ethier2009markov}, we find by the previous that $\pi$ is causal if and only if, for all $t \in \{2,\ldots, N\}$ and $a_t \in C_b(\X_{1:t} \times \Y_{1:t-1})$,
        \[
            \int a_t(x_{1:t},y_{1:t-1}) \, d\pi_{1:t,1:t-1} = \int \int a_t(x_{1:t}, y_{1:t-1}) \, d\mu^{x_{1:t-1}}_t \, d\pi_{1:t-1,1:t-1}.
        \]
        Rearranging the terms in the displayed equation above yields the assertion.
    \end{proof}

\begin{lemma}
	\label{lem:all sets of couplings are compact}
	Let $\mu \in \mathcal P(\X)$ and $\nu \in \mathcal P(\Y)$.
	Then $\Pi_\bullet(\mu,\nu)$ is compact w.r.t.\ weak convergence in $\mathcal P(\X \times \Y)$.
\end{lemma}

\begin{proof}
	As the Borel sets of $\X$ and $\X^\mu$ resp.\ $\Y$ and $\Y^\nu$ coincide, we have that
	$\mu$ and $\nu$ are Borel measures on $\mathcal P(\X^\mu)$ and $\mathcal P(\Y^\nu)$ respectively.
	Thus, $\Pi(\mu,\nu)$ is a compact subset of $\mathcal P(\X^\mu \otimes \Y^\nu)$.
	By Lemma \ref{lem:causality_char} we get
	\[
	\Pi_\bullet(\mu,\nu)
	=
	\left\{
	\pi \in \Pi(\mu,\nu)
	\, : \,
	\int s \, d\pi = \int s \, dP,
	\quad
	\forall s \in \tilde{\mathcal H}_c
	\right\}.
	\]
	Since $\tilde{\mathcal H}_\bullet \subseteq C_b(\X^\mu \otimes \Y^\nu)$ the identity above yields compactness of $\Pi_\bullet(\mu,\nu)$ as a subset of $\mathcal P(\X^\mu \otimes \Y^\nu)$, hence, it is compact in $\mathcal P(\X \times \Y)$.
\end{proof}

\bibliographystyle{imsart-number}
\bibliography{numbib.bib}

\end{document}